\documentclass{amsart}
\usepackage{graphicx, amssymb}
\usepackage[alphabetic,y2k,lite]{amsrefs}
\usepackage{subfigure}
\usepackage{tikz}
\usetikzlibrary{shapes.geometric}
\usetikzlibrary{calc}
\usetikzlibrary{scopes}
\usetikzlibrary{decorations.markings}

\tikzset{
every picture/.style={line width=0.8pt, >=stealth,
                       baseline=-3pt,label distance=-3pt},
dotnode/.style={fill=black,circle,minimum size=2.5pt, inner sep=1pt, outer
sep=0},
morphism/.style={circle,draw,thin, inner sep=1pt, minimum size=15pt,
                 scale=0.8},
small_morphism/.style={circle,draw,thin,inner sep=1pt,
                       minimum size=10pt, scale=0.8},
coupon/.style={draw,thin, inner sep=1pt, minimum size=18pt,scale=0.8},
regular/.style={densely dashed},
dual/.style={dashed, draw=blue, text=black},
boundary/.style={thick,  draw=blue, text=black},
overline/.style={preaction={draw,line width=2mm,white,-}},
drinfeld center/.style={>=stealth,green!60!black, double distance=1pt,text=black},
hopf/.style={>=stealth,blue!60!black, double
distance=1pt,text=black}, 
cell/.style={fill=black!10},
subgraph/.style={fill=black!30},
midarrow/.style={postaction={decorate},
                 decoration={
                    markings,% switch on markings
                    mark=at position #1 with {\arrow{>}},
                 }},
midarrow/.default=0.5
}

\newcommand{\pentagon}{
 \node[dotnode] (v1) at (72:1) {};
 \node[dotnode] (v2) at (0:1) {};
 \node[dotnode] (v3) at (-72:1) {};
 \node[dotnode] (v4) at (-144:1) {};
 \node[dotnode] (v5) at (144:1) {};
}
\newtheorem*{theorem*}{Theorem}
\newtheorem{theorem}{Theorem}[section]
\newtheorem{lemma}[theorem]{Lemma}

\newtheorem{corollary}[theorem]{Corollary}

\theoremstyle{definition}
\newtheorem{definition}[theorem]{Definition}

\newtheorem{example}[theorem]{Example}

\theoremstyle{remark}
\newtheorem{remark}[theorem]{Remark}

\numberwithin{equation}{section}
\newcommand{\firef}[1]{Figure~{\rm\ref{#1}}}
\newcommand{\thref}[1]{Theorem~{\rm\ref{#1}}}

\newcommand{\leref}[1]{Lemma~{\rm\ref{#1}}}
\newcommand{\coref}[1]{Corollary~{\rm\ref{#1}}}

\newcommand{\deref}[1]{Definition~{\rm\ref{#1}}}

\newcommand{\seref}[1]{Section~{\rm\ref{#1}}}

\newcommand{\st}{\; | \;}                               %%  such that
\newcommand{\ov}{\overline}
\newcommand{\del}{\partial}
\newcommand{\<}{\langle}
\renewcommand{\>}{\rangle}
\newcommand{\xxto}{\xrightarrow}              % long arrow
\newcommand{\dotimes}{\otimes\dots\otimes}    
\newcommand{\one}{\mathbf{1}}
\renewcommand{\i}{{\mathrm{i}}}   % imaginary unit
\newcommand{\CC}{\mathbb{C}}       % complex numbers
\newcommand{\N}{\mathcal{N}}       % combinatorial manifold
\newcommand{\DD}{\mathcal{D}}      % dimension of category
\newcommand{\C}{\mathcal{C}}      % category 
\newcommand{\Chat}{\widehat{\mathcal{C}}}      % category 
\newcommand{\M}{\mathcal{M}}      % combinatorial manifold
\newcommand{\A}{\mathcal{A}}      % category 
\newcommand{\Vect}{\mathcal{V}ec}  % vector spaces 
\newcommand{\ee}{\mathbf{e}}       % oriented edge
\newcommand{\oo}{\mathbf{o}}       % orientaion
\newcommand{\VV}{\mathbf{V}}       % boundary condition
\newcommand{\YY}{\mathbf{Y}}       % boundary condition
\newcommand{\Rbar}{{\overline{R}}} % dual hopf algebra
\newcommand{\HK}{\mathcal{H}_K}    % Kitaev model space
\newcommand{\HA}{\mathcal{H}_A}    % Kitaev model space
\newcommand{\LL}{\mathcal{L}}      % excited space
\renewcommand{\hbar}{{\bar h}}
\newcommand{\al}{\alpha}
\newcommand{\be}{\beta}
\newcommand{\ga}{\gamma}
\newcommand{\Ga}{\Gamma}
\newcommand{\de}{\delta}
\newcommand{\De}{\Delta}

\newcommand{\la}{\lambda}

\newcommand{\ph}{\varphi}

\newcommand{\Si}{\Sigma}

\newcommand{\eps}{\varepsilon}
\renewcommand{\th}{\theta}
\newcommand{\Hs}{H^{str}}
\newcommand{\HsD}{H^{str}_\De}
\newcommand{\Hhs}{\hat{H}^{str}}

\newcommand{\HTV}{H_{TV}}
\newcommand{\ZTV}{Z_{TV}}

\DeclareMathOperator{\Rep}{Rep}

\DeclareMathOperator{\Irr}{Irr}
\DeclareMathOperator{\id}{id}

\DeclareMathOperator{\Hom}{Hom}

\DeclareMathOperator{\End}{End}
\DeclareMathOperator{\tr}{tr}

\DeclareMathOperator{\im}{Im}

\DeclareMathOperator{\ev}{ev} % evaluation morphism
\DeclareMathOperator{\Obj}{Obj}

\DeclareMathOperator{\VGr}{VGraph}

\begin{document}
\title{Kitaev's Lattice Model and Turaev-Viro TQFTs}
\author{Benjamin Balsam}
   \address{Department of Mathematics, Stony Brook University, 
            Stony Brook, NY 11794, USA}

    \email{balsam@math.sunysb.edu}
    \urladdr{http://www.math.sunysb.edu/\textasciitilde balsam/}
\author{Alexander Kirillov, Jr.}
   \address{Department of Mathematics, Stony Brook University,
            Stony Brook, NY 11794, USA}

    \email{kirillov@math.sunysb.edu}
    \urladdr{http://www.math.sunysb.edu/\textasciitilde kirillov/}

\begin{abstract}
In this paper, we examine Kitaev's lattice model for an arbitrary complex,
semisimple Hopf algebra. We prove that this model gives the same
topological invariants as Turaev-Viro theory. Using the description of
Turaev-Viro theory as an extended TQFT, we prove that the excited states of
the Kitaev model correspond to Turaev-Viro theory on a surface with
boundary. 
\end{abstract}
\maketitle
\section*{Introduction}
In \cite{kitaev}, Kitaev introduced a series of quantum codes on a surface
which are suitable for fault-tolerant quantum computation. The best-known
of these is the famous toric code, which is based on $\mathbb{Z}_2$. The
toric code is simple to describe, but is not general enough to allow for
universal quantum computation. \footnote{This is related to the fact that
$\mathbb{Z}_2$ is abelian.} One may
consider a similar model based on any finite group $G$ or more generally,
any finite-dimensional semisimple Hopf algebra (see
\ocite{buerschaper}).\footnote{It is possible to
consider the situation in even more generality. The category of
representations of a semisimple Hopf algebra forms a spherical category,
but not every spherical category arises this way. One may consider lattice
models which start with an arbitrary spherical category, in this
 paper, we will consider categories of representations.}. Given a group
$G$,  Kitaev constructs a Hilbert space on a triangulated surface 
and a Hamiltonian, whose ground state is a topological invariant of the
surface. The excited states of this Hamiltonian form a mathematical model
of \textit{anyons}, particles which live in two-dimensions and which have
been proposed as a model for fault-tolerant quantum computing. The Kitaev
model was studied extensively in \ocite{buerschaper}, where the
authors present the model in the general Hopf algebra setting.

The string-net model was introduced by Levin and Wen in the context of
condensed matter physics \ocite{levin-wen}; in a different language (and in greater generality), this model was also described by  Walker in  \ocite{Walk}. An excellent exposition of this model and its applications to quantum
computing is \ocite{kuperberg}, and a careful mathematical description can
be found in \ocite{stringnet}, where the author proves that the string-net
model is isomorphic to Turaev-Viro theory. 

The relationship between the work of Levin-Wen and Kitaev was discussed in
\ocite{buerag}. Finally, a thorough analysis of Kitaev's model for a finite
group is presented in \ocite{bombdel}, which includes a very detailed
description of so-called \textit{ribbon operators}. It should be noted that  these papers are geared toward physicists and are somewhat difficult to 
read for mathematicians.

The main goal of this paper is to describe Kitaev's model in a 
fashion that is easily understandable to mathematicians and to relate it to
Turaev-Viro invariants. We begin with the construction of Kitaev's model for a general (semisimple) Hopf algebra $R$.  We carefully review the construction of a Hilbert space and Hamiltonian from  \ocite{buerschaper} and prove that on a
closed surface $\Sigma$, the ground state of this model is isomorphic to
that of the string-net model (and hence the Turaev-Viro space) on
$\Sigma$. This part is essentially a reformulation of known results; we hope, however, that our exposition will be more accessible to mathematicians. 

We then examine excited states of the Hamiltonian. These correspond to
higher eigenstates of the Hamiltonian, and  decompose into a
\textit{protected} space which is a topological invariant of $\Sigma$ and
certain local excitation spaces, which correspond to irreducible
representations of $D(R)$, the Drinfeld double of $R$. We show that the
excited states of the Kitaev model correspond to the Turaev-Viro and
string-net models for surfaces with boundary (see \ocite{balsam-kirillov},
\ocite{stringnet}). This part of the paper is new.

%\section{Mathematical Preliminaries}
\section{Hopf Algebras} \label{s:hopf}
\subsection{Basic definitions}
Throughout the paper, we denote by $R$ a finite dimensional Hopf algebra
over $\CC$ with 
\begin{itemize}
\item multiplication $\mu_R\colon R\otimes R\to R$, 
\item unit $\eta_R\colon \CC\to R$
\item comultiplication $\Delta_R\colon R \to R\otimes R$
\item counit $\epsilon_R\colon R\to \CC$
\item antipode  $S_R\colon R\to R$ 
\end{itemize}
We will drop  the subscript $R$ when there is no ambiguity.

We will use the Sweedler notation, writing $\De (x)=x'\otimes x''$,
$\De^2(x)=x'\otimes x''\otimes x'''$, etc.;  summation will be implicit in
these formulas. If the number of factors is large, we will also use the
alternative notation writing $\De^{(n-1)}(x)=x^{(1)}\otimes
x^{(2)}\dotimes x^{(n)}$. 

We will denote by $R^*$ the dual Hopf algebra. We will use Greek letters
$\al, \be,\dots$ for elements of $R^*$. We will also use the Sweedler
notation for comultiplication in $R^*$, writing
$\De_{R^*}(\al)=\al'\otimes\al''$; thus, 
\begin{equation}\label{e:comult_dual}
  \<\al'\otimes\al'', x_1\otimes x_2\>=\<\al, x_1x_2\>
\end{equation}
where $\<\, \>$ stands for the canonical pairing $R^*\otimes R\to \CC$.

From now on, we will also assume that $R$ is semisimple. The following
theorem shows that in fact this condition can be replaced by one of
several equivalent conditions.

\begin{theorem} \label{t:larson} \cite{LR}
  Let $R$ be a finite-dimensional Hopf algebra over a field of
  characteristic zero. Then the following are equivalent: 
  \begin{enumerate}
    \item $R$ is semisimple.
    \item $R^*$ is semisimple.
    \item $S_R^2 = \id$.
  \end{enumerate}
\end{theorem}
% In fact, requiring that $R$ be semisimple is enough:
% \begin{theorem}
% Let $R$  be a semisimple Hopf algebra. Then $R$ is finite-dimensional.
% \end{theorem}
% For a proof of this theorem, see (\cite{DNR}, p.189). Although it is
%redundant, we will often remind the reader that $R$ is finite-dimensional
%and semisimple.
\subsection{Haar integral}
Let $R$ be as described above. Then we have a distinguished element in $R$
called the Haar integral which is defined by the following conditions:
\begin{enumerate}
\item $ hx = xh = \epsilon_{R}(x)h\quad \text{ for all }x \in R$.
\item $ h^2 = h$.
\end{enumerate}
 The following theorem lists important properties of the Haar integral.
\begin{theorem} \label{t:haar}
  Let $R$ be a semisimple, finite-dimensional Hopf algebra. Then
  \begin{enumerate}
    \item $h$ exists and is unique.
    \item $S(h) = h$.
    \item For any $n$, the element
    $$
    h_n=\De^{(n-1)}(h)\in R^{\otimes n}
    $$ 
    is cyclically invariant. In particular,  for $n=2$, we have
    $\De^{op}(h) = \De(h)$. 
\end{enumerate}
\end{theorem}

\begin{lemma}\label{l:haar}
 The Haar integral acts by 1 in the trivial representation of $R$ and 
 by 0 in  any irreducible non-trivial representation.
\end{lemma}
\begin{proof}
  This is an immediate consequence of the definition of Haar integral.
\end{proof}

\begin{corollary}\label{c:haar}
  Let $W_1,\dots, W_n$ be finite-dimensional representations of $R$. Then 
$h_n=\De^{n-1}(h)$ acts in $W_1\dotimes W_{n}$ by projection onto the
invariant subspace  
$$
\{w\in W_1\dotimes W_n\st \De^{n-1}(x)w=\eps(x)w\}\simeq \Hom_R(\one,
W_1\dotimes W_n)
$$
\end{corollary}

\subsection{Representations}
Since $R$ is semisimple, any representation of $R$ is completely 
reducible. We will denote by $V_i$,  $i\in I$, a set of representatives of 
isomorphism classes of irreducible representations of $R$; we will also
use the notation $d_i=\dim V_i$. In particular, 
the trivial one-dimensional representation of $R$ will be denoted 
$$
\one=V_0.
$$

We will frequently use the notation 
$$
\<W_1,\dots, W_n\>=\Hom_R(\one, W_1\dotimes W_n)
\simeq\{w\in W_1\dotimes W_n\st \De^{n-1}(x)w=\eps(x)w\}
$$
(compare with \coref{c:haar}).

Given a representation $V$ of $R$, we can define the dual representation
$V^*$. As a vector space $V^*$ is the  ordinary dual space to $V$. The
action of $R$ is defined by
\begin{equation} \label{e:dualrep}
\<x\al, v\> = \<\al, S(x)v\>
\end{equation}
where $\al \in V^*, v\in V$,  and $x \in R$. Note that since $S^2=\id$, the
usual vector space isomorphism $V^{**}\simeq V$ is an isomorphism of
representations; thus, $V^{**}$ is canonically isomorphic to $V$ as a 
representation of $R$.

Since  a dual of an irreducible representation is irreducible, we have an
involution $\vee\colon I\to I$ such that $V_i^*\simeq V_{i^\vee}$. This
isomorphism is not canonical; however, one does have a canonical
isomorphism  
\begin{equation}\label{e:Vcheck}
V_i\otimes V_i^*\simeq V_{i^\vee}^*\otimes V_{i^\vee}.
\end{equation}

For any two representations $V, W$ of $R$, we denote by $\Hom_R(V,W)$ the 
space of $R$-morphisms from $V$ to $W$. We have a non-degenerate pairing 
\begin{equation}\label{e:pairing}
\Hom_R(V,W)\otimes \Hom_R(V^*,W^*)\to \CC
\end{equation}
given by 
$$
\<\ph,\psi\>=(\one\to (V\otimes V^*)\xxto{\ph\otimes \psi}
               W\otimes W^* \to \one)
$$

By semisimplicity, we have a 
canonical  isomorphism
\begin{equation} \label{e:isom}
 R \cong \bigoplus_{i\in I} \End_{\mathbb{C}}(V_i)
  = \bigoplus_{i} V_i \otimes V_{i}^{*}.
\end{equation}
It is convenient to write the Hopf algebra structure of
$R$ in terms of this isomorphism. 
\begin{lemma}\label{l:hopfrewritten}
Under isomorphism \eqref{e:isom}, multiplication, comultiplication, unit,
counit, and antipode of $R$ are given by  
\begin{itemize}
\item \textbf{Multiplication}: $\mu_{R}\colon V_{i} \otimes V^*_{i} \otimes
V_{j} \otimes V^*_{j} \xrightarrow{1 \otimes ev \otimes 1} \delta_{i,j}
V_{i} \otimes V^*_{i}$, where $ev\colon V_i^*\otimes V_i\to \CC$ is the
evaluation map. 
\item \textbf{Comultiplication}:
Let $\ph_{\al}$ be a basis for $\Hom_{R}(V_{i}, V_{j} \otimes
V_{k})$ and $\varphi^{\alpha} \in \Hom_{R}(V^*_{i}, V^*_{k} \otimes
V^*_{j})$ the dual basis with respect to the pairing given by \eqref{e:pairing}.
Then
\begin{align*}
\Delta_{R}\colon   \displaystyle \bigoplus_{i} V_{i} \otimes V^*_{i}
&\xrightarrow{\displaystyle \sum d_i\varphi_{\alpha} \otimes
\varphi^{\alpha}} \displaystyle \bigoplus_{j,k} V_{j} \otimes V_{k} \otimes
V^*_{k} \otimes V^*_{j}\\
& \xrightarrow{\tau_{23}\circ \tau_{34}}
\displaystyle \bigoplus_{j,k} V_{j} \otimes V^*_{j} \otimes V_{k} \otimes
V^*_{k}
\end{align*}
\item \textbf{Unit}: $\eta_{R} = \displaystyle \bigoplus_{i} \displaystyle
\sum_{j=1}^{dim V_i} v_{j} \otimes  v^*_{j} = \displaystyle \sum_{i}
\id_{V_i}$, where $\{v_j\}$ is a basis for $V_i$ and $\{v^*_{j}\}$ is the
dual basis.
\item \textbf{Counit}: For $a \otimes b \in V_{i} \otimes V^*_{i}$, 
$\epsilon_{R}(a \otimes b) = \delta_{i,0}b(a)$ 
\item \textbf{Antipode}: If $a \otimes b \in V_{i} \otimes V^{*}_{i}$, then 
$$
S(a\otimes b) = b \otimes a\in V_i^*\otimes V_i\simeq V_{i^\vee}\otimes V_{i^\vee}^*
$$
\textup{(}see \eqref{e:Vcheck}\textup{)}. 
\end{itemize}
\end{lemma}

In this language, the Haar integral is given by the canonical element $1
\in V_{0} \otimes V^*_{0}$.

\subsection{Graphical calculus}
We will frequently use graphical presentation of morphisms between
representations of $R$. We will use the same conventions as in
\ocite{stringnet}, representing a morphism $\ph\colon W_1\otimes \dots
W_k\to W'_1\otimes \dots\otimes W'_l$ by a tangle with $k$ strands labeled
$W_1,\dots, W_k$ at the top and $l$ strands labeled $W'_1,\dots, W'_l$ at
the bottom. We will also use the usual cap and cup tangles to represent
evaluation and coevaluation morphisms.

\subsection{Dual Hopf algebra}
Given a semisimple Hopf algebra $R$, we will define the following 
version of the dual Hopf algebra
\begin{equation}\label{e:Rbar}
\Rbar = (R^{op})^*
\end{equation}
where  $R^{op}$ denotes the algebra $R$ with opposite multiplication. 

Since comultiplication in $\Rbar$ is opposite to comultiplication in
$R^*$,  notation $\al', \al''$ is ambiguous. We adopt the following
convention: notation $\al', \al''$ (or, equivalently, $\al^{(1)}$,
$\al^{(2)}$,\dots) always refers to comultiplication in $R^*$. Thus,
comultiplication in $\Rbar$ is given by 
$$
\De_{\Rbar}(\al)=\al''\otimes \al'.
$$

Note that $\ov{\ov{R}}$ is canonically isomorphic to $R$ as a vector space
but as a Hopf algebra, has opposite  multiplication and comultiplication.
Thus, the map 
$$
S\colon R\to \ov{\ov{R}}
$$
is an isomorphism of Hopf algebras.

Note that by \thref{t:larson}, $\Rbar$ is also semisimple and thus has a 
unique Haar integral. We will denote by 
\begin{equation}
\bar h\in \Rbar
\end{equation}
the Haar integral of $\Rbar$. 

\begin{lemma}
    Let $R$ be a semisimple Hopf algebra. Then the Haar integral of
    $\Rbar$ is given by 
    $$
    \<\bar h, x\>=\frac{1}{\dim R}\tr_R(x)
    $$
    where $\tr_R(x)=\sum d_i \tr_{V_i}(x)$ is the trace of action of $x$
    in the (left or right) regular representation. 
\end{lemma}

We can also rewrite the Haar integral of $\Rbar$ in terms of the isomorphism 
$R\simeq \bigoplus V_i\otimes V_i^*$ (see \eqref{e:isom}).

\begin{lemma} \label{l:hbar}
Let $x_k=v_k\otimes w_k\in V_{i_k}\otimes V_{i_k}^*$; 
using  isomorphism \eqref{e:isom}, each $x_k$ can be considered as
element of $R$. Then 
$$
\<\bar h, x_1\dots x_n\>= \begin{cases}
          \frac{d_i}{\dim R}
           \<v_1, w_n\> \<w_1, v_2\>\dots \<w_{n-1}, v_n\>, &\quad 
                                        i_1=i_2=\dots=i_n=i,\\
           0& \text{otherwise}
           \end{cases}
$$
\textup{(}see \firef{f:hbar}\textup{)}.
\end{lemma}
\begin{figure}[ht]
$$
\<\bar h, (v_1\otimes w_1)\cdots  (v_n\otimes w_n)\>=\frac{d_i}{\dim R}
\begin{tikzpicture}[baseline=-0.5cm]
\node[above] (v1) at (0,0) {$v_1$};
\node[above] (w1) at (0.5,0) {$w_1$};
\node[above] (v2) at (1.5,0) {$v_2$};
\node[above] (w2) at (2.0,0) {$w_2$};
\node[above] (vn) at (5,0) {$v_n$};
\node[above] (wn) at (5.5,0) {$w_n$};
\draw (w1)-- +(0,-0.2) arc(-180:0:0.5) -- +(0,0);
\draw (w2)-- +(0,-0.2) arc(-180:0:0.5) -- +(0,0);
\draw (vn)-- +(0,-0.2) arc(0:-180:0.5) -- +(0,0);
\draw (v1) ..controls +(0,-2) and +(0,-2).. (wn);
\end{tikzpicture}
$$
\caption{The Haar integral of $\Rbar$}\label{f:hbar}
\end{figure}

\subsection{Regular action}\label{s:regular}
With have two obvious actions of $R$ on itself, called left and right
regular actions: 
\begin{itemize}
\item $L_x\colon y\mapsto  xy$ 
\item $R_x\colon y\mapsto  = yS(x)$ 
\end{itemize}
Note that these actions commute. 

In a similar way, the dual Hopf algebra $\Rbar$ can also be endowed with
two commuting actions of $\Rbar$.

Using \eqref{e:dualrep}, we can also define  two actions of $R$ on $\Rbar$:
\begin{itemize}
\item  $L_x^*\colon \la\mapsto x.\la$, where  $\<x.\la,
y\>=\<\la,  S(x)y\>$
\item $R_x^*\colon \la\mapsto \la.S(x)$, where  $\<\la.x,
y\>=\<\la,  yS(x)\>$
\end{itemize}
and two actions of $\Rbar$ on on $R$:
\begin{itemize}
\item  $L_\al^*\colon x\mapsto \al.x:=\<\al,S(x')\>x''$,
so that $\<\la, \al.x\>=\<S(\al)\la, x\>$
\item $R_\al^*\colon x\mapsto x.S(\al)\>=x'\<\al, x''\>$, so that  
$\<\la,x.\al\>=\<\la S(\al),  x\>$
\end{itemize}

Note that notation $x.\al$ is ambiguous, as it can mean $L_x^*(\al)$ or
$R_{S(\al)}^*(x)$. In most cases the meaning will be clear from the
context.

Note that we can also define left and right action of the Hopf algebra
$\ov{\ov R}$ on $\Rbar$. It is easy to see that these actions are given
by
\begin{itemize}
\item  $L_t\colon \la\mapsto \la.t$, where $t\in \ov{\ov R}$ is considered
as an element of $R$ via trivial vector space isomorphism.
\item $R_t\colon \la\mapsto S(t).\la$, where $t\in \ov{\ov R}$ is
considered
as an element of $R$ via trivial vector space isomorphism.
\end{itemize}

We will use these actions (together with left and right regular 
actions of $R$ on itself)  repeatedly throughout the rest of the paper. 
All the operators we will discuss can be defined in terms of them.

\begin{lemma}\label{l:hbar2}
  Let $\hbar$ be the Haar integral of $\Rbar$. Consider the left regular
action of $\hbar$ on $R^{\otimes n}$:
\begin{equation}
\begin{aligned}
    L^*_{\hbar}\colon R^{\otimes n}&\to R^{\otimes n}\\
    x_n\dotimes x_1&\mapsto \hbar^{(n)}.x_n\dotimes \hbar^{(1)}x_1\\
                   &=\<\hbar,S(x'_n\dots x'_1)\> x''_n\dotimes x''_1
                   &=\<\hbar,x'_n\dots x'_1\> x''_n\dotimes x''_1
\end{aligned}
\end{equation}
\textup{(}recall that comultiplication in $\Rbar$ is given by
$\De^{(n-1)}\al=\al^{(n)}\dotimes \al^{(1)}$\textup{)}. 

Then, after identifying each copy of $R$ with $\bigoplus V_i\otimes
V_i^*$ \textup{(}see \eqref{e:isom}\textup{)}, $L_{\hbar}^*$ is given by the following
picture:
\begin{align*}
&\sum_{i_1, \dots, i_n,j_1, \dots, j_n, k}    \frac{d_{i_1}\dots d_{i_n}
d_k}{\dim R}\sum_{\al, \be, \dots}\\
 &   \begin{tikzpicture}
\coordinate (ytop) at (0,1); \coordinate (ybot) at (0,-3);
%\draw[drinfeld center](ytop)--(ybot);
%%%%%%%%%%
\node[morphism] (v1) at (1,0) {$\scriptstyle\ph_\al$};
\node[morphism] (w1) at (1.5,0) {$\scriptstyle\ph^\al$};
\node[above] at (v1|-ytop) {$i_n$};
\node[above] at (w1|-ytop) {$i^*_n$};
\node[below] at (v1|-ybot) {$j_n$};
\node[below] at (w1|-ybot) {$j^*_n$};
\draw (v1|-ytop) -- (v1) -- (v1|-ybot);
\draw (w1|-ytop) -- (w1) -- (w1|-ybot);
%%%%%%%%%%
\node[small_morphism] (v2) at (3,0) {$\scriptstyle\ph_\be$};
\node[small_morphism] (w2) at (3.5,0) {$\scriptstyle\ph^\be$};
\draw (v2|-ytop) -- (v2) -- (v2|-ybot);
\draw (w2|-ytop) -- (w2) -- (w2|-ybot);
\draw[->] (w1) .. controls (2.25, -0.8).. (v2) node[pos=0.5, above] {$k$};
\draw[->] (w2) .. controls (4.25, -0.8).. (5,0) node[pos=0.5, above] {$k$};
%%%%%%%%%%
\node[small_morphism] (vn) at (9,0) {};
\node[small_morphism] (wn) at (9.5,0) {};
\draw (vn|-ytop) -- (vn) -- (vn|-ybot);
\draw (wn|-ytop) -- (wn) -- (wn|-ybot);
\node[above] at (vn|-ytop) {$i_1$};
\node[above] at (wn|-ytop) {$i^*_1$};
\node[below] at (vn|-ybot) {$j_1$};
\node[below] at (wn|-ybot) {$j^*_1$};
\draw[->] (7.5,0) .. controls (8.25, -0.8).. (vn) node[pos=0.5, above] {$k$};
%%%%%%%%%%
\coordinate (l) at (0,-1);
\coordinate (r) at (10.5,-1);
\draw[<-] (v1) 
          .. controls +(225:0.5) and +(90:0.5) ..
      (l)
          .. controls +(-90:1.5) and +(-90:1.5) ..
      (r) node[pos=0.5, above] {$k$}
          .. controls +(90:0.5) and +(-45:0.5) ..
      (wn);
\end{tikzpicture}
\end{align*}
where $\ph_\al, \ph^\al$ are as in \leref{l:hopfrewritten}.

Note that the crossings in the picture are just permutation of factors (there is no braiding in the category $\Rep R$) and the whole map is not a morphism of representations. 
\end{lemma}
\begin{proof}
    Follows by combining the formula for multiplication and
comultiplication in $R$ (\leref{l:hopfrewritten}) and \leref{l:hbar}. 
\end{proof}

\subsection{Drinfeld double}

Given a finite-dimensional semisimple Hopf algebra $R$, there is a
well-known way to construct from it a quasitriangular Hopf algebra $D(R)$.
This new Hopf algebra, called the \textit{Drinfeld Double} of $R$, has
numerous applications in representation theory and physics. We will review
some basic details of the construction. For a much more detailed
description and proofs see, e.g.  \cite{Kassel},
\cite{etingof-schiffmann}. 

Let $R$ be a semisimple, finite-dimensional Hopf algebra; as before, let 
$\Rbar = (R^{op})^*$.

\begin{theorem}
  The following operations define on the vector space $R\otimes \Rbar$ a
  structure of a Hopf algebra. This Hopf algebra will be denoted $D(R)$ and
  called {\em Drinfeld double} of $R$. 

  \begin{enumerate}
    \item Multiplication: 
     $$
      (x\otimes \al )\cdot(y \otimes \be ) 
        = x  y'' \otimes  \al^y \be,
     $$
     where  $\al^y \in \Rbar$ is  defined by 
     \begin{equation}\label{e:double2}
        \< \al^y, z\>=\<\al,y'''zS^{-1}(y')\>
     \end{equation}
    \item Unit: $1_{D(R)} =   1_{R} \otimes 1_{\Rbar}$
    \item Comultiplication: 
      $x\otimes \al \mapsto (x'\otimes \al'') \otimes (x''\otimes \al')$
    \item Counit: 
      $x\otimes \al 
       \mapsto \epsilon_{R}(x) \epsilon_{\Rbar}(\al)$
    \item Antipode: 
      $S(x\otimes \al)=S(\al)S(x)=S(x'') S(\al)^{S(x)}$
  \end{enumerate}
\end{theorem}
\begin{remark}
This definition follows \ocite{etingof-schiffmann} and differs slightly
from the one given in \ocite{Kassel}, where $D(R)$ is defined as
$\Rbar\otimes R$. However, it is not difficult
to show that these definitions are  equivalent.
\end{remark}

Recall (see \cite{Kassel}) that $D(R)$ has  a canonical quasitriangular 
structure, with $R$-matrix $R=\sum x_\al\otimes x_\al$, where $x_\al,
x^\al$ are dual bases in $R, \Rbar$ respectively. Thus, the category
of finite-dimensional representations of
$D(R)$ has a structure of a braided ribbon category. Moreover, it is known
that the category of representations of $D(R)$ is in fact equivalent to the
so-called Drinfeld Center of the category of representations of $R$ (see
\ocites{muger1, muger2}). In particular, for any representation $Y$ of
$D(R)$ and a representation $V$ of $R$, we have a functorial isomorphism 
\begin{equation}\label{e:Rmatrix}
\begin{aligned}
    V\otimes Y&\mapsto Y\otimes V\\
    v\otimes y&\mapsto \sum_\al x^\al y\otimes  x_\al v
\end{aligned}
\end{equation}
This map (or sometimes its inverse) will be called the {\em
half-braiding}.

\begin{lemma}\label{l:haar_double}
  Let $R$ be a semisimple, finite-dimensional Hopf algebra. Then $D(R)$ is
  semisimple with Haar integral given by  
  \begin{equation}
    h_{D(R)} = h_R \otimes h_{\Rbar}.
  \end{equation}
  Moreover, both $h,\bar h $ are central in $D(R)$.
\end{lemma}

\subsection{Action of $D(R)$ on $R\otimes R$}
In this section, we define an action of $D(R)$ on $R\otimes R$. This will
be used in the future.  

\begin{lemma} \label{l:comlemma}
  For $a \in R$, $\al \in \Rbar$, define the operators 
   $p_{a}\colon R \otimes R \to R \otimes R$ and 
   $q_{\al}\colon  R \otimes R \to R \otimes R$ by 
  \begin{align*}
    p_{a}(u \otimes v) &= a'u \otimes vS(a'')\\
    q_{\al}(u \otimes v) &= \<\al, S(u'v'))u''\otimes v''
                             =\al''.u\otimes \al'.v.
  \end{align*}
  %(see \firef{}). 
  Then these operators satisfy the commutation relations of $D(R)$: 
  the map 
  \begin{align*}
             D(R)&\to \End(R\otimes R)\\
     a\otimes \al&\mapsto p_a q_\al
  \end{align*}
  is a morphism of algebras.
\end{lemma}
\begin{proof}
  This follows by explicit computation (\cite{buerschaper}), using the
  following formulas: 
  \begin{align*}
    \al.(xy)&=x'' (\al_x . y)=(_y\al.x) y'',\qquad \text{where}\\
        \<\al_x, z\>&=\<\al, zS(x')\>\\
        \<_y \al,z\>&=\<\al, S(y')z\>
  \end{align*}
\end{proof}

%%%%%%%%%%%%%%%%%%%%%%%%%%%%%%%%%%%%%%%%%
\section{Kitaev's Model}
%%%%%%%%%%%%%%%%%%%%%%%%%%%%%%%%%%%%%%%%%
In this section, we look at Kitaev's lattice model. This model is a
generalization of the well-known toric code; we get a theory for any
finite-dimensional semisimple Hopf algebra $R$  over $\CC$. We begin with a
compact, oriented surface $\Si$ with a fixed cell decomposition $\De$.
\footnote{Some papers begin instead with a surface $\Sigma$ with an
embedded graph $\Gamma$. This is clearly equivalent data; the graph
$\Gamma$ corresponds to the 1-skeleton of the cell decomposition.}. We will
assign to $(\Si, \De)$ a finite-dimensional Hilbert space $\HK(\Si)$ and
introduce a Hamiltonian consisting of local operators. The ground state of
this Hamiltonian is useful for quantum computation. We will later show see
that this ground state can be identified with the Turaev-Viro vector space
$Z_{TV}(\Si)$. This obviously implies that the ground state is a
topological invariant of $\Si$: in particular, it does not depend on the
cell decomposition $\De$.  

From now on, we fix a choice of  finite-dimensional, semisimple Hopf
algebra $R$.

\subsection{Crude Hilbert space}\label{s:crudespace}
Given a compact oriented surface $\Si$ with a cell decomposition $\De$, we
denote by $E$ the set of (unoriented) edges of $\De$. Then for any choice
$\oo$ of orientation of each edge of $\Si$,   we define the space 
\begin{equation} 
\HK(\Sigma, \De,\oo) = \bigotimes_{E} R 
\label{e:hilbert}
\end{equation} 
We will graphically represent a vector $\bigotimes x_e\in \HK$ by writing
the vector $x_e$ next to each edge $e$. 

So defined vector space depends on the choice of orientation. However, in
fact vector spaces coming from different  orientations are canonically
isomorphic. Namely, if $\oo$ and $\oo'$ are two orientations that
differ by reversing orientation of a single edge $e$, we identify 
\begin{equation}\label{e:switch}
\begin{aligned}
\HK(\oo)&\to \HK(\oo')\\
   x_e&\mapsto S(x_e)
\end{aligned}
\end{equation}
(see \firef{f:switch}). Note that since $S^2=\id$, this isomorphism is well
defined. This shows that all spaces $\HK(\oo)$ for different choices of
orientation are canonically isomorphic to each other; thus, we will drop
the choice of orientation from our formulas writing just $\HK(\Si,\De)$.  

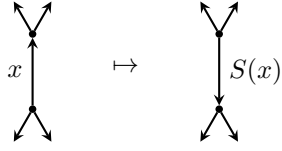
\begin{figure}[ht]
$$
\begin{tikzpicture}[baseline=0.5cm]
\node[dotnode] (A) at (0,0){};
\node[dotnode] (B) at (0,1){};
\draw[->] (A)--(B) node[pos=0.5, left]{$x$};
\draw[->] (A)-- +(-60:0.5);
\draw[->] (A)-- +(240:0.5);
\draw[->] (B)-- +(60:0.5);
\draw[->] (B)-- +(120:0.5);
\end{tikzpicture}
\qquad
\mapsto
\qquad 
\begin{tikzpicture}[baseline=0.5cm]
\node[dotnode] (A) at (0,0){};
\node[dotnode] (B) at (0,1){};
\draw[<-] (A)--(B) node[pos=0.5, right]{$S(x)$};
\draw[->] (A)-- +(-60:0.5);
\draw[->] (A)-- +(240:0.5);
\draw[->] (B)-- +(60:0.5);
\draw[->] (B)-- +(120:0.5);
\end{tikzpicture}
$$
\caption{The antipode allows us to identify the Hilbert spaces obtained via
any two choices of edge orientation.} 
\label{f:switch}
\end{figure}

The Hilbert space $\HK$ is clearly not a topological invariant; in
particular, its dimension depends on number of edges in $\De$.

\subsection{Vertex and plaquette operators}
We now define a collection of operators on $\HK$; in the next section, we
will use them to construct the Hamiltonian on $\HK$. As before, we fix a
closed oriented surface $\Si$ and a choice of cell decomposition $\De$. 

\begin{definition} \label{d:site}
Let $(\Si, \De)$ be a surface with a cell decomposition. A \textit{site}
$s$ is a pair $(v, p)$ where $v$ is a vertex of $\De$ and $p$ is an
adjacent plaquette (face). 
\end{definition}
A typical site is shown in \firef{f:site}.
We will depict a site as a green line connecting a vertex to the center of an
adjacent plaquette. Equivalently, if we superimpose the dual lattice, a
site connects a vertex in the lattice to an adjacent vertex in the dual
lattice.  
\begin{figure}[ht]
\begin{tikzpicture}
\pentagon
\draw (v1)--(v2)--(v3)--(v4)--(v5)--(v1);
\draw[green](0,0)--(v1);
\node[below] at (0,0) {$p$};
\node[above right] at (v1) {$v$};
\end{tikzpicture}

\caption{The site $s=(v,p)$ is drawn as a green line connecting $v$ and the
center of $p$} \label{f:site} 
\end{figure}
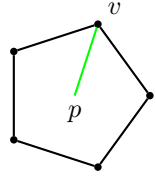

 At each vertex $v$, we have a natural counterclockwise cyclic ordering of
edges incident to $v$. Similarly, given a plaquette (2-cell)  $p$, we have 
the {\bf clockwise} cyclic ordering of edges on $\partial p$. 

\begin{definition}\label{d:vertex_op}
Given a site  $s=(v,p)$ of the cell decomposition $\De$ and an element
$a\in R$, the vertex operator $A^a_{v,p}\colon \HK(\Si,\De)\to \HK(\Si,\De)
$ is defined by  
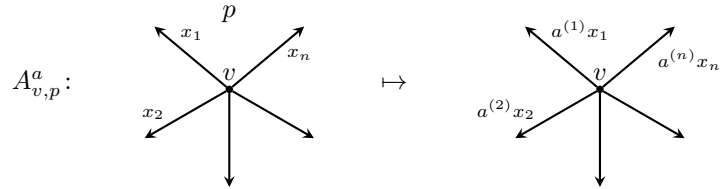
\begin{center}
$$A^a_{v,p}\colon 
\qquad
\begin{tikzpicture}[baseline=0cm]
\node[dotnode] (v) at (0,0){}; \node[above] at (v) {$v$};
\node at (0,1) {$p$};
\draw[->] (v)-- +(40:1.3) node[pos=0.7, below right=-2pt]{$\scriptstyle
x_n$};
\draw[->] (v)-- +(-30:1.3);
\draw[->] (v)-- +(-90:1.3);
\draw[->] (v)-- +(-150:1.3) node[pos=0.7, above left=-2pt]{$\scriptstyle
x_2$};
\draw[->] (v)-- +(140:1.3) node[pos=0.7, above right=-2pt]{$\scriptstyle
x_1$};
\end{tikzpicture}
\qquad
\mapsto
\qquad
\begin{tikzpicture}[baseline=0cm]
\node[dotnode] (v) at (0,0){}; \node[above] at (v) {$v$};
\draw[->] (v)-- +(40:1.3) node[pos=0.7, below right=-2pt]{$\scriptstyle
a^{(n)}x_n$};
\draw[->] (v)-- +(-30:1.3);
\draw[->] (v)-- +(-90:1.3);
\draw[->] (v)-- +(-150:1.3) node[pos=0.7, above left=-2pt]{$\scriptstyle
a^{(2)}x_2$};
\draw[->] (v)-- +(140:1.3) node[pos=0.7, above right=-2pt]{$\scriptstyle
a^{(1)}x_1$};
\end{tikzpicture}
$$
\end{center}
where the edges incident to $v$ are indexed {\bf counterclockwise} starting
from $p$.
\end{definition}
In the definition above, the  edges incident to $v$ are all pointing away
from the vertex.  It is easy to see, using \eqref{e:switch},
that if any edge is oriented towards  the vertex, then
left action would be replaced by the right action: instead of
$a^{(i)}x_i$, we would have
$x_iS(a^{(i)})$.

In a similar way, one defines the plaquette operators. 

\begin{definition}\label{d:plaquette_op}
  Given  a site  $s=(v,p)$ of the cell decomposition
  $\De$ and an element $\al\in \Rbar$,   the plaquette  operator
  $B^{\al}_{p,v}\colon \HK(\Si,\De)\to \HK(\Si,\De) $ is defined by  
  \begin{align*}
   B^\al_{p,v}\colon 
  \qquad
  \begin{tikzpicture}[baseline=0cm]
  \node at (0, 0) {$p$};
  \pentagon
  \node[above right] at (v1) {$v$};
  \draw[->] (v1) -- (v2) node[pos=0.5, right]{$\scriptstyle x_n$};
  \draw[->] (v2) -- (v3);
  \draw[->] (v3) -- (v4); 
  \draw[->] (v4) -- (v5)  node[pos=0.5, left]{$\scriptstyle x_2$}; 
  \draw[->] (v5) -- (v1) node[pos=0.5, above left=-2pt]{$\scriptstyle
x_1$};
  \end{tikzpicture}
  \qquad
  &\mapsto
  \begin{tikzpicture}[baseline=0cm]
  \node at (0, 0) {$p$};
  \pentagon
  \node[above right] at (v1) {$v$};
  \draw[->] (v1) -- (v2) node[pos=0.5, right]{$\scriptstyle \al^{(n)}.
x_n$};
  \draw[->] (v2) -- (v3);
  \draw[->] (v3) -- (v4); 
  \draw[->] (v4) -- (v5) node[pos=0.5, left]{$\scriptstyle \al^{(2)}.
x_2$}; 
  \draw[->] (v5) -- (v1) node[pos=0.5, above left=-2pt]
    {$\scriptstyle \al^{(1)}.x_1$};
  \end{tikzpicture}
  \\
  &=
  \<\al, S(x'_n\dots x'_1)\>\quad 
  \begin{tikzpicture}[baseline=0cm]
  \node at (0, 0) {$p$};
  \pentagon
  \node[above right] at (v1) {$v$};
  \draw[->] (v1) -- (v2) node[pos=0.5, right]{$\scriptstyle x''_n$};
  \draw[->] (v2) -- (v3);
  \draw[->] (v3) -- (v4); 
  \draw[->] (v4) -- (v5) node[pos=0.5, left]{$\scriptstyle  x''_2$}; 
  \draw[->] (v5) -- (v1) node[pos=0.5, above left=-2pt]
    {$\scriptstyle x''_1$};
  \end{tikzpicture}
\end{align*}
where $\al.x$ stands for left action of $\Rbar$ on $R$ as defined in
\seref{s:regular}. 
\end{definition}

In the definition above, the edges surrounding $p$ are all given a
{\bf clockwise} orientation (even though the indices go counterclockwise).
It is easy to see, using \eqref{e:switch},
that if any edge is oriented counterclockwise, then left action would be
replaced by the right action: instead of $\al^{(i)}.x_i$, we would have
$x_i.S(\al^{(i)})$.

\begin{theorem}\label{t:com_relations}\par\noindent
    \begin{enumerate}
        \item If $v,w$ are distinct vertices, then the operators 
        $A_v^a$, $A^b_{w}$ commute for any $a, b\in R$. 
        
        Similarly, if $p,q$ are distinct plaquettes, then the operators 
        $B_p^\al$, $B_q^\be$ commute for any $\al, \be\in \Rbar$. 
        \item If $v$, $p$ are not incident to one another, then operators $A_v^a$,
$B_p^\al$ commute. 
        
        \item For a given site $s=(v,p)$,  the operators $A_{v,p}^a$,
             $B_{p,v}^\al$ satisfy
            the commutation relations of Drinfeld double of $R$: the map
      \begin{align}
          \rho_s\colon D(R)&\to \End(\HK(\Si,\De))\\
          a\otimes\al&\mapsto A_v^a B_p^\al
      \end{align}
      is an algebra morphism. 
 \end{enumerate}
\end{theorem}
\begin{proof}
 \begin{enumerate}
\item  The operators $A_v$, $A_w$ obviously commute if the edges incident
to $v$ and those incident to $w$ are disjoint. We therefore assume that $v$
and $w$ are adjacent, i.e. at least one edge connects them. Clearly, we
need only to check that the actions of $A_{v}$, $A_{w}$ commute on their
common support. Suppose such an edge $e$ is oriented so that it points from
$v$ to $w$. Then $A_{v}$ acts on the corresponding copy of $R$  via the
left regular representation, and $A_{w}$ acts on $e$ via the right regular
representation. These are obviously commuting actions. The proof for
plaquette operators is similar. 
\item Obvious.
\item Follows from the following generalization of  \leref{l:comlemma},
proof of which we leave to the reader. 
\begin{lemma} \label{l:comlemma2}
  Let $X$ be a representation of $R$, and $Y$ -- a representation of $\Rbar$.  
  For $a \in R$, $\al \in \Rbar$, define the operators 
   $p_{a}, q_\al\in \End( R \otimes X\otimes Y \otimes R)$ by 
  \begin{align*}
    p_{a}(u \otimes x\otimes y \otimes  v) &= a'u \otimes a'' x\otimes y \otimes  v S(a''')\\
    q_{\al}(u \otimes x\otimes y \otimes  v) &
     = \al'''.u \otimes x\otimes \al''. y \otimes  \al'.v 
  \end{align*}
  %(see \firef{}). 
  Then these operators satisfy the commutation relations of $D(R)$: 
  the map 
  \begin{align*}
             D(R)&\to \End(R\otimes X\otimes Y\otimes R)\\
     a\otimes \al&\mapsto p_a q_\al
  \end{align*}
  is a morphism of algebras.
\end{lemma}

\end{enumerate}
    
\end{proof}

\subsection{Duality}
The $A$ and $B$ projectors are dual to one another in the following sense.
Consider a \textit{dual} theory, in which we begin with the dual cell
decomposition  $\De^*$ with edge orientation inherited from $\De$ as shown
in \firef{f:edgeconvention} and the dual Hopf algebra $\Rbar$. 
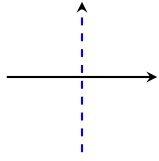
\begin{figure}[ht]
\begin{tikzpicture}[baseline=0cm]
\draw[->] (-1,0) -- (1,0);
\draw[dual, ->] (0,-1)--(0,1);
\end{tikzpicture}
\caption{The convention for orienting edges of the dual graph is shown
above. Here the solid black edge is from the original cell decomposition 
$\De$ and the dashed blue arrow belongs to the dual one $\De^*$.} 
\label{f:edgeconvention}
\end{figure}
 
 We get the  Hilbert space $\ov{\HK}$ which may be identified with $\HK^*$
using the evaluation pairing $\ev\colon \Rbar \otimes R \to \CC$. Note that
the vertices in $\De^*$ correspond to plaquettes in $\De$ and vice versa.
The following lemma shows that the vertex operators from one theory
correspond naturally to the plaquette operators from the other. 
%\begin{lemma} \label{l:a2b}
%Under the natural pairing $\< , \>$ of $\mathcal{H}^*$ and $\mathcal{H}$, we have 
%\begin{enumerate}
%\item $\<A_{v} \psi, \varphi \> = \<\psi, B_{p}\varphi\>$
%\item $\<B_{p} \psi, \varphi \> = \<\psi, A_{v}\varphi\>$
%\end{enumerate}
%where $\psi \in \mathcal{H}^*, \varphi \in \mathcal{H}$ and in each equation, $v$ and $p$ consist of a vertex in one decomposition and its corresponding plaquette in the dual decomposition.
%\end{lemma}
%FIXME: proof
%A generalization of this result is proved in \leref{l:a2b2}. 

\begin{lemma} \label{l:a2b2}
Under the natural pairing $\< , \>$ of $\ov{\HK}$ and $\HK$, we have 
\begin{equation}
\begin{aligned}
    \<y, B_s^\al x\>&=\<A^{S(\al)}y,x\>\\
    \<y, A_s^a x\>&=\<B^a_s y, x\>
\end{aligned}
\end{equation}
where $x\in \HK$, $y\in \ov{\HK}$,  $a \in R, \al\in \Rbar$ and $s$ is a
site (both in $\De$ and $\De^*$). 
\end{lemma}
\begin{proof}
Let $s$ be a site; label edges of $\De, \De^*$ around $s$ as shown in
\firef{f:duality1}. 

\begin{figure}[ht]
\begin{tikzpicture}
\pentagon
\draw[->] (v1)--(v2) node[pos=0.8,above right=-2pt] {$\scriptstyle x_n$};
\draw[->] (v2)--(v3);
\draw[->] (v3)--(v4);
\draw[->] (v4)--(v5) node[pos=0.8,left=-2pt] {$\scriptstyle x_2$};
\draw[->] (v5)--(v1) node[pos=0.8,above left=-2pt] {$\scriptstyle x_1$};
\draw[green] (v1)--(0,0);
\draw[dual, ->] (0,0) -- (36:2cm) node[pos=0.8,above left=-2pt]
{$\scriptstyle y_n$};
\draw[dual, ->] (0,0) -- (108:2cm) node[pos=0.8,above right=-2pt]
{$\scriptstyle y_1$};
\draw[dual, ->] (0,0) -- (180:2cm) node[pos=0.8,above =-2pt]
{$\scriptstyle y_2$};
\draw[dual, ->] (0,0) -- (-36:2cm);
\draw[dual, ->] (0,0) -- (-108:2cm);
\end{tikzpicture}
\caption{Proof of duality formulas}\label{f:duality1}
\end{figure}
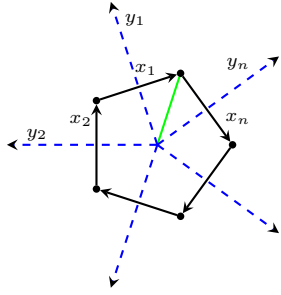
Then 
\begin{align*}
\<y, B^\al_s x\>&=\<y_1,\al^{(1)}.x_1\>\dots\<y_n, \al^{(n)}x_n\>\\
&=\<S(\al^{(1)})y_1, x_1\>\dots \<S(\al^{(n)})y_n, x_n\>\\
&=\<S(\al)^{(n)}y_1, x_1\>\dots \<S(\al)^{(1)}y_n, x_n\>\\
&=\<A^{S(\al)}_sy,x\>
\end{align*}
(recall that comultiplication in $\Rbar$ is given by $\De^{(n-1)}_\Rbar
\be=\be^{(n)}\dotimes\be^{(1)}$.)

The second identity is proved similarly. 
\end{proof}

\subsection{The groundspace}

Let $\HK(\Si,\De)$ be as in \seref{s:crudespace}. Consider the following 
 special case of the vertex and plaquette operators:
 \begin{equation}\label{e:vertex2}
 \begin{aligned}
 A_v&=A_v^h\\
 B_p&=B_p^{\bar h}
 \end{aligned}
 \end{equation}
 where $h\in R$, $\bar h\in \Rbar$ are the Haar integrals of $R, \Rbar$.
 
Note that since $\De^n (h)$ is cyclically invariant (see \thref{t:haar}),
 the operator $A_v$ only depends on the vertex $v$ and not on the choice of
the adjacent plaquette $p$ (which was used before to construct the linear
ordering of the edges adjacent to $v$); similarly,  $B_p$ only depends on
the choice of $p$. 

Using these operators, we define the Hamiltonian $H\colon \HK\to \HK$ by 
\begin{equation} \label{e:hamiltonian}
H = \displaystyle \sum_{v}(1 - A_{v}) + \displaystyle \sum_{p}(1 - B_{p})
\end{equation}

The most important property of this Hamiltonian is that it consists of {\em
commuting} operators. 

\begin{theorem} \label{t:commute}
\par\noindent
  \begin{enumerate}
    \item All operators $A_v$, $B_p$ commute with each other.
    \item Each of these operators is idempotent: $A_v^2=A_v$, $B_p^2=B_p$.
  \end{enumerate}
\end{theorem}
\begin{proof}
Immediately follows from \thref{t:com_relations} and $h^2=h$, $h$ is
central  (\thref{t:haar}). 
\end{proof}

The Hamiltonitan \eqref{e:hamiltonian} is a sum of these local projectors
and since they all commute, $H$ is diagonalizable.  
\begin{definition}\label{d:groundspace}
The {\em ground state} $K_{R}(\Si,\De)$ of Kitaev's model  is the zero
eigenspace of $H$:  
$$
  K_{R}(\Si,\De) = \{x \in \HK(\Si,\De)| Hx = 0\}.
$$
\end{definition}
It is easy to see that $x \in K_{R}(\Sigma)$ iff $A_{v}x = B_{p}x
= x$ for every vertex $v$ and plaquette $p$. 

We will show below that up to a canonical isomorphism, the groundspace does
not depend on the choice of the cell decomposition.

%%%%%%%%%%%%%%%%%%%%%%%%%%%%%%%%%
\section{Turaev--Viro and Levin-Wen models}
%%%%%%%%%%%%%%%%%%%%%%%%%%%%%%%%%
In this section, we give an overview of two other theories: Turaev--Viro
and Levin-Wen (stringnet) models. All results of this section are  known and
given here just for the readers convenience. 

We will mostly follow the approach and notation of our earlier papers
\ocites{balsam-kirillov, stringnet}, to which the reader is referred for
more detail and references.  

Throughout the section, we let  $\A$ be a spherical fusion category, i.e. a
fusion category together with a functorial isomorphism $V\simeq V^{**}$
satisfying appropriate properties. We will denote by $\{V_i, i\in I\}$ the
set of representatives of isomorphism classes of simple objects in $\A$,
and by $d_i\dim V_i$ the categorical dimension of $V_i$. We will also use
the notation $\DD=\sqrt{\sum d_i^2}$.

Note that for every semisimple finite-dimensional Hopf algebra $R$, the
category $\A=\Rep(R)$ of finite-dimensional representations of $R$ is a
spherical fusion category, and the notation $V_i, d_i$ agree with the
notation of \seref{s:hopf}. In this case, $\DD^2=\dim R$. 

\subsection{Turaev--Viro model}
Let $\A$ be a spherical fusion category as above. Then one can define a
3-dimensional  TQFT $Z_{TV}^\A$, called
the Turaev--Viro model; it was originally defined in \ocite{TV} and
generalized to arbitrary spherical categories by Barrett and Westbury
\ocite{barrett}. In particular, for
any closed oriented surface $\Si$ this theory gives a vector space
$Z_{TV}^\A(\Si)$, defined as follows.  

First, we choose a cell decomposition $\De$ of $\Si$. A coloring of edges
of $\De$ is a choice, for every {\bf oriented}  edge $\ee$ of $\De$ of a
simple object $l(\ee)$ so that $l(\ov{\ee})=l(\ee)^*$.

We define the state space  
$$
\HTV(\Si,\De)=\bigoplus_{l}\bigotimes_C H(C,l)
$$
where $l$ is a coloring of edges of $\De$, $C$
is a 2-cell of $\De$, and
\begin{equation}
H(C,l)=\<l(e_1), l(e_2),\dots,
l(e_n)\>,\qquad
 \del C=e_1\cup e_2\dots\cup e_n
\end{equation}
where the edges $e_1,\dots, e_n$ are taken in the counterclockwise order on
$\del C$ as shown in \firef{f:state_space1}.
\begin{figure}[ht]
%%%%%%%%%%%%%%
\begin{tikzpicture}
\node at (0,0) {$C$};
\foreach \i in {1,...,5}
   {
    \pgfmathsetmacro{\u}{\i*72}
    \pgfmathsetmacro{\v}{\u+72}
    \draw[->] (\u:1.2cm)--(\v:1.2cm) node[pos=0.5,auto=right]{$X_\i$};
    }
\end{tikzpicture}
%%%%%%%%%%%%%%
\qquad\qquad $H(C)=\<X_1, \dots,  X_n\>
=\Hom_\A(\one, X_1\otimes\dots\otimes X_n)$
\caption{State space for a cell}\label{f:state_space1}
\end{figure}
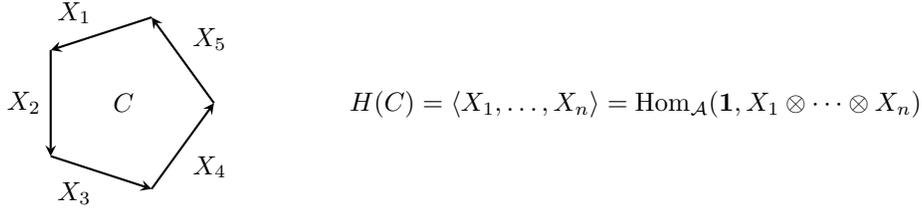

Next, given a cobordism $M$ between two surfaces $\Si, \Si'$ with cell
decompositions, one can define an operator $Z(M)\colon \HTV(\Si, \De)\to
\HTV(\Si',\De')$; it is defined using a cell decomposition of $M$ but can
be shown to be independent of the choice of the decomposition (see
\ocite{balsam-kirillov}*{Theorem 4.4}). In
particular, taking $M=\Si\times I$, we get an operator $Z(\Si\times
I)\colon \HTV(\Si,\De)\to\HTV(\Si,\De)$ which can be shown to be a
projector. We now define the Turaev--Viro space associated to $\Si$ as
\begin{equation}\label{e:ZTV}
\ZTV^\A(\Si, \De)=\im(Z(M\times I)).
\end{equation}
It can be shown that for any two cell decompositions $\De,\De'$ of the same
surface $\Si$, we have a canonical isomorphism $\ZTV^\A(\Si, \De)\simeq
\ZTV^\A(\Si, \De')$ (see \ocite{balsam-kirillov}*{}); thus, this space is
determined just by the surface $\Si$. Therefore, we will omit $\De$ in the
notation, writing just $\ZTV^\A(\Si)$.

\subsection{Stringnet model}
There is also another way of constructing a vector space associated to an
oriented closed surface $\Si$; this construction was introduced in the
papers of Levin and Wen \ocites{levin-wen}. We will refer to it as
stringnet model (or sometimes  as Levin-Wen model ). In
this section we give an overview of this model, following the conventions
of \ocite{stringnet}. 

In this model, we again begin with a spherical fusion category $\A$ and
consider colored graphs $\Ga$ on $\Si$. Edges of the graph should be
oriented and colored by objects of $\A$ (not necessarily simple); vertices
are colored by morphisms $\psi_v\in \Hom_\A(\one, V_1\otimes\dots\otimes
V_n)$, where $V_i$ are colors of edges incident to $v$ taken in
counterclockwise order and with  outward orientation (if some edges
come with inward orientation, the corresponding $V_i$ should be replaced by
$V_i^*$).  

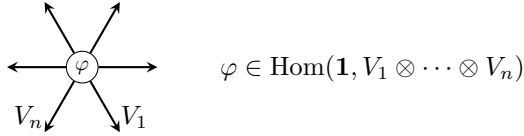
\begin{figure}[ht]
%%%%%%%%
\begin{tikzpicture}
\node[morphism] (ph) at (0,0) {$\ph$};
\draw[->] (ph)-- +(240:1cm) node[pos=0.7, left] {$V_n$} ;
\draw[->] (ph)-- +(180:1cm);
\draw[->] (ph)-- +(120:1cm);
\draw[->] (ph)-- +(60:1cm);
\draw[->] (ph)-- +(0:1cm);
\draw[->] (ph)-- +(-60:1cm) node[pos=0.7, right] {$V_1$};
\end{tikzpicture}
\qquad 
$\ph\in\Hom(\one, V_1\dotimes V_n)$
%%%%%%%%
\caption{Labeling of colored graphs}\label{f:coloring}
\end{figure}

We will follow the conventions of \ocite{stringnet}; in particular, if a
graph contains a pair of vertices, one with outgoing edges labeled
$V_1,\dots, V_n$ and the other with edges labeled $V_n^*,\dots, V_1^*$ ,
and the vertices  are labeled by the same letter $\al$  (or $\be$, or
\dots) it will stand for summation over the dual bases:
\begin{equation}\label{e:summation_convention}
%%%%%%%%
\begin{tikzpicture}
\node[morphism] (ph) at (0,0) {$\al$};
\draw[->] (ph)-- +(240:1cm) node[pos=0.7, left] {$V^{*}_1$} ;
\draw[->] (ph)-- +(180:1cm);
\draw[->] (ph)-- +(120:1cm);
\draw[->] (ph)-- +(60:1cm);
\draw[->] (ph)-- +(0:1cm);
\draw[->] (ph)-- +(-60:1cm) node[pos=0.7, right] {$V^{*}_n$};
\node[morphism] (ph') at (3,0) {$\al$};
\draw[->] (ph')-- +(240:1cm) node[pos=0.7, left] {$V_n$} ;
\draw[->] (ph')-- +(180:1cm);
\draw[->] (ph')-- +(120:1cm);
\draw[->] (ph')-- +(60:1cm);
\draw[->] (ph')-- +(0:1cm);
\draw[->] (ph')-- +(-60:1cm) node[pos=0.7, right] {$V_1$};
\end{tikzpicture}
%%%%%%%%
\quad \mathrel{\mathop:}= \sum_\al\quad
%%%%%%%%
\begin{tikzpicture}
\node[morphism] (ph) at (0,0) {$\ph_\al$};
\draw[->] (ph)-- +(240:1cm) node[pos=0.7, left] {$V^{*}_1$} ;
\draw[->] (ph)-- +(180:1cm);
\draw[->] (ph)-- +(120:1cm);
\draw[->] (ph)-- +(60:1cm);
\draw[->] (ph)-- +(0:1cm);
\draw[->] (ph)-- +(-60:1cm) node[pos=0.7, right] {$V^{*}_n$};
\node[morphism] (ph') at (3,0) {$\ph^\al$};
\draw[->] (ph')-- +(240:1cm) node[pos=0.7, left] {$V_n$} ;
\draw[->] (ph')-- +(180:1cm);
\draw[->] (ph')-- +(120:1cm);
\draw[->] (ph')-- +(60:1cm);
\draw[->] (ph')-- +(0:1cm);
\draw[->] (ph')-- +(-60:1cm) node[pos=0.7, right] {$V_1$};
\end{tikzpicture}
%%%%%%%%
\end{equation}
where $\ph_\al\in \<V_1,\dots, V_n\>$, $\ph^\al\in \<V_n^*,\dots, V_1^*\>$
are dual bases with respect to pairing \eqref{e:pairing}.

We then define the stringnet space 
$$
\Hs(\Si)=\text{Formal linear combinations of  colored graphs on $\Si$}/
         \text{Local relations}
$$
Local relations come from embedded disks in $\Si$; the precise definition can
be found in \ocite{stringnet}. Here we only give one local relation which
will be useful in the future:
\begin{equation}\label{e:local_rel}
\sum_{i\in \Irr(\A)} d_i 
%%%%%%%%%%%%
\begin{tikzpicture}
\node[morphism] (ph) at (0,-0.5) {$\al$};
\node[morphism] (psi) at (0,0.5) {$\al$};
\node at (0,1.1) {$\dots$};
\node at (0,-1.1) {$\dots$};
\draw[->] (ph)-- +(-110:1cm) node[pos=1.0,left,scale=0.8]
{$V_1$};
\draw[->] (ph)-- +(-70:1cm) node[pos=1.0,right,scale=0.8]
{$V_n$};
\draw[<-] (psi)-- +(110:1cm) node[pos=1.0,left,scale=0.8]
{$V_1$};
\draw[<-] (psi)-- +(70:1cm) node[pos=1.0,right,scale=0.8]
{$V_n$};
\draw[<-] (ph) -- (psi) node[pos=0.5,left,scale=0.8] {$i$};
\end{tikzpicture}
%%%%%%%%%%%
=
%%%%%%%%%%%
\begin{tikzpicture}
\node at (0,0) {$\dots$};
\draw[<-] (-0.3,-1)-- (-0.3,1) node[pos=0.5,left,scale=0.8]
{$V_1$};
\draw[<-] (0.3,-1)-- (0.3,1) node[pos=0.5,right,scale=0.8]
{$V_n$};
\end{tikzpicture}
\end{equation}

The following result has been stated in a number of papers; a rigorous
proof can be found in \ocite{stringnet}. 

\begin{theorem} \label{t:TVLW}
Let $\A$ be a spherical fusion category and $\Si$ -- a closed oriented
surface. Then one has a canonical  isomorphism
$\ZTV^\A(\Si)\simeq\Hs(\Si)$.
\end{theorem}
In fact, we will need a more detailed version of the theorem above. Namely,
let $\De$  be a cell decomposition of $\Si$. Let $\Si-\De^0$ be the
surface
with punctures obtained by removing from $\Si$ all vertices of $\De$ and
let $\HsD=\Hs(\Si-\De^0)$ be the corresponding stringnet space. Then
one has the following results.

\begin{theorem}\label{t:TVLW2}
  \par\noindent
  \begin{enumerate}
    \item The natural map $\HsD\to \Hs$ induces an isomorphism
$$
\Hs\simeq \im(B^s)=\{x\in\HsD\st B^{str}_p x=x\ \forall p\}\subset
\HsD
$$
where $B^{str}=\prod_p B^{str}_p$, $p$ runs over the set of vertices of
$\De$ and $B^{str}_p\colon\HsD\to\HsD$ is the operator which adds to a
colored graph a small loop around puncture $p$ as shown below.
\textup{(}The superscript $str$ is introduced to avoid confusion with the
plaquette operators $B_p$ in Kitaev's model; relation between the two
operators is clarified below.\textup{)}
\begin{figure}[ht]
$$
\sum_i\frac{d_i}{\DD^2}\quad
%%%%%%%%%%%%%
\begin{tikzpicture}
\node[dotnode, label=-45:$p$] at (0,0) {};
\draw (0,0) circle(0.5cm);
\node[right] at (30:0.6) {$i$};
\end{tikzpicture}
%%%%%%%%%%%%%
$$
\caption{Operator $B^{str}_p$}\label{f:B_p}
\end{figure}

    \item One has a natural isomorphism $\HTV(\Si,\De)\simeq \HsD$
    \item Under the isomorphism of the previous part, the operator
associated to the cylinder  $Z_{TV}(\Si\times I)\colon
\HTV(\Si,\De)\to\HTV(\Si,\De)$ is
identified with the projector $B^{str}=\prod_p B^{str}_p\colon
\HsD\to\HsD$.
  \end{enumerate}
\end{theorem}

The proof of this theorem can be found in \ocite{stringnet}; obviously, it
implies \thref{t:TVLW}. Note that the isomorphism constructed in the proof
requires a non-trivial choice of normalizations; see \ocite{stringnet} for
details.

%%%%%%%%%%%%%%%%%%%%%%%%%%%%%%%%%%%%%%%%%%%%
\section{The main theorem: closed surface}\label{s:main1}
%%%%%%%%%%%%%%%%%%%%%%%%%%%%%%%%%%%%%%%%%%%%

In this section, we prove the first main result in the paper, identifying
the ground space $K_{R}(\Si,\De)$ of Kitaev model with the vector space
$Z_{TV}^\A$ of the Turaev--Viro TQFT with the category $\A=\Rep(R)$.

\begin{theorem} \label{t:main1}
Let $R$ a finite-dimensional semisimple Hopf algebra. Then for any 
closed, oriented surface $\Si$ with a cell decomposition $\De$ one has a
canonical isomorphism $K_{R}(\Si,\De) \cong Z_{TV}^\A(\Si)$, where
$Z_{TV}^\A$  is the Turaev-Viro TQFT based on the category $\A = \Rep(R)$.
\end{theorem}

For example, on the sphere $S^2$, the ground state is one-dimensional, or
\textit{non-degenerate} in physics terminology. 

The proof of this theorem occupies the rest of this section. For brevity,
we will denote the Hilbert space of Kitaev model just by $\HK$, dropping
$\Si,\De$ from the notation. 

Recall that $K_R\subset \HK$ was defined as the ground space of the
Hamiltonian. We begin by introducing an intermediate vector space
$\HA$ such that $K_R\subset \HA\subset \HK$. Namely, we let 
$$
\HA=\ker(\sum(1-A_v)) =\{x\in \HK\st A_v x=x\ \forall \ v\}\subset
\HK
$$
where $A_v$ are vertex operators \eqref{e:vertex2} and the sum is over all
vertices $v$ of $\De$. 

Since $A_v$, $B_p$ commute, the $B_p$ operators preserve subspace $\HA\subset
\HK$. The following equality is obvious from the definitions:
\begin{equation}
    K_R=\{x\in \HA\st B_p x= x \ \forall\ p\}
\end{equation}
where $p$ ranges over all 2-cells of $\De$.

We can now formulate the first lemma relating Kitaev's model with the
Turaev--Viro TQFT.

\begin{lemma}\label{l:main1}
    One has a natural isomorphism 
    $$
    \HA(\Si,\De)\simeq \HTV(\Si,\De^*)
    $$
    where $\De^*$ is the dual cell decomposition.
\end{lemma}
\begin{proof}

Recall that we have an isomorphism $R\simeq \bigoplus V_i\otimes V_i^*$
(see \eqref{e:isom}). Using this isomorphism, we can  give an equivalent
description  of the vector space $\HK$. Namely,
let us denote by $E^{or}$  the set of oriented edges of $\De$, i.e. pairs
$\ee=(e,\text{orientation of } e)$; for such an oriented edge $\ee$, we
denote by $\bar{\ee}$ the edge with opposite orientation.  

Then we can rewrite the definition of $\HK$ as follows:
\begin{equation}\label{e:HK2}
\HK=\bigoplus_{l} \bigotimes_{\ee} l(\ee)
\end{equation}
where the sum is over all  colorings of   edges  of
$\De$ and tensor product is over all {\em oriented} edges $\ee$
of $\De$; thus, every unoriented edge $e$ appears in this tensor product
twice, with opposite orientations. We will illustrate a vector
$v=\bigotimes v_\ee$ by drawing two oriented half-edges in place of every
(unoriented) edge $e$ and writing the corresponding vector $v_\ee$ next to
each half-edge, as shown in \firef{f:halfedge}.

\begin{figure}[ht]
$$
\begin{tikzpicture}[baseline=0.5cm]
\node[dotnode] (A) at (0,0){};
\node[dotnode] (B) at (0,1){};
\draw[->] (A)--(B) node[pos=0.5, left]{$x=v\otimes w$};
\draw[->] (A)-- +(-60:0.5);
\draw[->] (A)-- +(240:0.5);
\draw[->] (B)-- +(60:0.5);
\draw[->] (B)-- +(120:0.5);
\end{tikzpicture}
\qquad
\leftrightarrow
\qquad
\begin{tikzpicture}[baseline=0.5cm]
\node[dotnode] (A) at (0,0){};
\node[dotnode] (B) at (0,1){};
\draw[->] (A)--(0,0.5) node[pos=0.5, left]{$v$};
\draw[<-] (0,0.5)--(B) node[pos=0.5, left]{$w$};
\draw[->] (A)-- +(-60:0.5);
\draw[->] (A)-- +(240:0.5);
\draw[->] (B)-- +(60:0.5);
\draw[->] (B)-- +(120:0.5);
\end{tikzpicture}
\qquad
\leftrightarrow
\qquad
\begin{tikzpicture}[baseline=0.5cm]
\node[dotnode] (A) at (0,0){};
\node[dotnode] (B) at (0,1){};
\draw[<-] (A)--(B) node[pos=0.5, right]{$S(x)=w\otimes v$};
\draw[->] (A)-- +(-60:0.5);
\draw[->] (A)-- +(240:0.5);
\draw[->] (B)-- +(60:0.5);
\draw[->] (B)-- +(120:0.5);
\end{tikzpicture}
$$
\caption{$\HK=\bigoplus_{l} \bigotimes_{\ee} V_{l(\ee)}$
}
\label{f:halfedge}
\end{figure}

Re-arranging the factors of \eqref{e:HK2}, we can write 
$$
\HK=\bigoplus_{l} \bigotimes_v \mathcal{H}_v
$$
where the product is over all vertices $v$ of the cell decomposition $\De$
and
$$
\mathcal{H}_v=l(e_1)\otimes \dots\otimes l(e_n)
$$
where the $l(e_1), \dots, l(e_n)$ are the colors of edges incident to $v$
taken in counterclockwise order with outgoing orientation.

In this language, the vertex operator $A_v$ acts on $\mathcal{H}_v$ by
$$
A_v(a_1\otimes \dots\otimes a_n)=h^{(1)}a_1\otimes\dots\otimes h^{(n)}a_n
$$

By \coref{c:haar}, we see that therefore the image of $\prod A_v$ is the
space
$$
\HA=\bigoplus_{l} \bigotimes_v \<l(e_1),\dots, l(e_n)\>
$$
where, as before,  $l(e_1), \dots, l(e_n)$ are the colors of edges incident
to $v$ taken in counterclockwise order with outgoing orientation.

Since vertices of $\De$ correspond to 2-cells of $\De^*$, this gives an
isomorphism 
$$
\theta\colon \HA\simeq
\bigoplus_{l}\bigotimes_{v}H_{TV}(C_v,l)=H_{TV}(\Si,\De^*)
$$
where $C_v$ is the 2-cell of $\De^*$ corresponding to vertex $v$ of $\De$.

However, for reasons that will become clear in the future, we will rescale
this isomorphism and define
$$
\tilde \theta=\sqrt{\mathbf{d}_l} \theta
$$
where, for any choice of simple coloring $l$  of edges, 
$\sqrt{\mathbf{d}_l}$ acts on $\bigotimes_{C_v} H(C_v,l)$ by multiplication
by the factor
$$
\prod_{e} \sqrt{d_{l(e)}}
$$
where the product is over all unoriented edges $e$ of $\De^*$.
\end{proof}

\firef{f:HA} shows the composition map
$\HA\xxto{\tilde\th} \HTV(\Si, \De^*)\simeq \Hs_{\De^*}$ (cf.
\thref{t:TVLW2}). 

\begin{figure}[ht]
\begin{align*}
\begin{tikzpicture}
\node at (0,0) {$\ast$};
\pentagon
\coordinate (m1) at ($(v1)!.5!(v2)$); %midpoint of a side
\coordinate (m2) at ($(v2)!.5!(v3)$);
\coordinate (m3) at ($(v3)!.5!(v4)$);
\coordinate (m4) at ($(v4)!.5!(v5)$);
\coordinate (m5) at ($(v5)!.5!(v1)$);
\draw[->](v1)--  +(72:0.5) node[pos=0.7, above left]{$\scriptstyle u_1$};
\draw[->](v5)--  +(144:0.5) node[pos=0.7, above right]{$\scriptstyle u_2$};
\draw[->](v4)--  +(-144:0.5) node[pos=0.7, below right]{$\scriptstyle u_3$};
\draw[->](v3)--  +(-72:0.5) node[pos=0.7, below left]{$\scriptstyle u_4$};
\draw[->](v2)--  +(0:0.5) node[pos=0.7, above]{$\scriptstyle u_5$};
\draw[->] (v1)--(m1) node[pos=0.5, above right=-2pt]{$\scriptstyle w_5$};
\draw[->] (v2)--(m1) node[pos=0.5, above right=-2pt]{$\scriptstyle v_5$};
\draw[->] (v2)--(m2) node[pos=0.5, below right=-2pt]{$\scriptstyle w_4$};
\draw[->] (v3)--(m2) node[pos=0.5, below right=-2pt]{$\scriptstyle v_4$};
\draw[->] (v3)--(m3) node[pos=0.5, below]{$\scriptstyle w_3$};
\draw[->] (v4)--(m3) node[pos=0.5, below]{$\scriptstyle v_3$};
\draw[->] (v4)--(m4) node[pos=0.5, left]{$\scriptstyle w_2$};
\draw[->] (v5)--(m4) node[pos=0.5, left]{$\scriptstyle v_2$};
\draw[->] (v5)--(m5) node[pos=0.5, above]{$\scriptstyle w_1$};
\draw[->] (v1)--(m5) node[pos=0.5, above]{$\scriptstyle v_1$};
\end{tikzpicture}
\quad 
&\mapsto
\dots d_{i_1}\dots d_{i_5}\dots 
\begin{tikzpicture}
\node at (0,0) {$\ast$};
 \node[morphism] (v1) at (72:1) {$\ph_1$};
 \node[morphism] (v2) at (0:1) {$\ph_5$};
 \node[morphism] (v3) at (-72:1) {$\ph_3$};
 \node[morphism] (v4) at (-144:1) {$\ph_3$};
 \node[morphism] (v5) at (144:1) {$\ph_2$};
\draw[->](v1)--  +(72:0.7) node[pos=0.7, above left]{$\scriptstyle j_1$};
\draw[->](v5)--  +(144:0.7) node[pos=0.7, above right]{$\scriptstyle j_2$};
\draw[->](v4)--  +(-144:0.7) node[pos=0.7, below right]{$\scriptstyle j_3$};
\draw[->](v3)--  +(-72:0.7) node[pos=0.7, below left]{$\scriptstyle j_4$};
\draw[->](v2)--  +(0:0.7) node[pos=0.7, above]{$\scriptstyle j_5$};
 \draw[->] (v1)--(v5) node[pos=0.5, above]{$\scriptstyle i_1$};
 \draw[->] (v5)--(v4)  node[pos=0.5, left]{$\scriptstyle i_2$};
 \draw[->] (v4)--(v3)  node[pos=0.5, below left]{$\scriptstyle i_3$};
 \draw[->] (v3)--(v2)  node[pos=0.5, below right]{$\scriptstyle i_4$};
 \draw[->] (v2)--(v1)  node[pos=0.5, above right]{$\scriptstyle i_5$};
\end{tikzpicture}\\
& v_1\in V_{i_1}, w_1\in V_{i_1}^*, \dots \\
& \ph_1=u_1\otimes v_1\otimes w_5, \dots
\end{align*}
\caption{Isomorphism $\HA\simeq\Hs_{\De^*}$. Asterisk $\ast$
shows the puncture obtained by removing a vertex of $\De^*$.}\label{f:HA}
\end{figure}

\begin{lemma}\label{l:main2}
  Under the isomorphism of \leref{l:main1}, the operator
$B=\prod_p B_p\colon  \HA\to \HA$ is identified with the operator
$Z_{TV}(\Si\times I)\colon \HTV(\Si,\De^*)\to\HTV(\Si,\De^*)$.
\end{lemma}
\begin{proof}
  We will prove it in the language of stringnets: combining isomorphism
$\HA\simeq \HTV(\Si,\De^*)$ (see \leref{l:main1}) and
$\HTV(\Si,\De^*)\simeq \Hs_{\De^*}$ (see
\thref{t:TVLW2}), we get an isomorphism $\HA\simeq \Hs_{\De^*}$,  and it
suffices to prove that under this isomorphism, the plaquette projector
$B_p$ of Kitaev model is identified with the projector $B_p^{str}$ of the
stringnet  model. To avoid complicated notation, we write explicitly the
proof in the case shown in \firef{f:HA}. 

Using \leref{l:hbar2}, we see
that  the projector $B_p$ of Kitaev model can be described as follows: if
$x\in \HA$ is as shown in \firef{f:HA}, then
$$
\tilde\theta (B_p x)=
\sum_{k, j_1, \dots, j_5}\frac{d_k}{\dim R}\ d_{j_1}\dots d_{j_5}\quad
\begin{tikzpicture}
\node at (0,0) {$\ast$};
 \node[morphism] (v1) at (72:2) {$\ph_1$};
 \node[morphism] (v2) at (0:2) {$\ph_5$};
 \node[morphism] (v3) at (-72:2) {$\ph_3$};
 \node[morphism] (v4) at (-144:2) {$\ph_3$};
 \node[morphism] (v5) at (144:2) {$\ph_2$};
 \node[small_morphism] (m1) at ($(v1)!.3!(v2)$){$\scriptstyle \al$}; %point 1/3 of the way 
   \node[small_morphism] (n1) at ($(v1)!.3!(v5)$) {$\scriptstyle \al$};
   %\coordinate (n1) at ($(v1)!.3!(v5)$); %point 1/3 of the way 
 \node[small_morphism] (m2) at ($(v2)!.3!(v3)$) {$\scriptstyle \eps$};
   \node[small_morphism] (n2) at ($(v2)!.3!(v1)$) {$\scriptstyle \eps$}; %point 1/3 of the way 
 \node[small_morphism] (m3) at ($(v3)!.3!(v4)$) {$\scriptstyle \de$};
   \node[small_morphism] (n3) at ($(v3)!.3!(v2)$) {$\scriptstyle \de$}; %point 1/3 of the way 
 \node[small_morphism] (m4) at ($(v4)!.3!(v5)$) {$\scriptstyle \ga$};
   \node[small_morphism] (n4) at ($(v4)!.3!(v3)$) {$\scriptstyle \ga$}; %point 1/3 of the way 
 \node[small_morphism] (m5) at ($(v5)!.3!(v1)$) {$\scriptstyle \be$};
   \node[small_morphism] (n5) at ($(v5)!.3!(v4)$) {$\scriptstyle \be$}; %point 1/3 of the way 
 \draw[->] (v1)--(n1) -- (m5) node[pos=0.5, above]{$\scriptstyle j_1$}
               -- (v5);
 \draw[->] (v2)--(n2) -- (m1) node[pos=0.5, above right]{$\scriptstyle j_5$}
               -- (v1);
 \draw[->] (v3)--(n3) -- (m2) node[pos=0.5, below right]{$\scriptstyle j_4$}
               -- (v2);
 \draw[->] (v4)--(n4) -- (m3) node[pos=0.5, below]{$\scriptstyle j_3$}
               -- (v3);
 \draw[->] (v5)--(n5) -- (m4) node[pos=0.5, left]{$\scriptstyle j_2$}
               -- (v4);
  \draw (m1) .. controls (72:1.1) .. (n1);    
  \draw (m2) .. controls (0:1.1) .. (n2);    
  \draw (m3) .. controls (-72:1.1) .. (n3);    
  \draw (m4) .. controls (-144:1.1) .. (n4);    
  \draw (m5) .. controls (144:1.1) .. (n5);    
  \node[anchor=base] at (72:1) {$k$};
  \node[anchor=base] at (0:1) {$k$};
  \node[anchor=base] at (-72:1) {$k$};
  \node[anchor=base] at (-144:1) {$k$};
  \node[anchor=base] at (144:1) {$k$};
 \end{tikzpicture}\\
$$

Now we can use local relation \eqref{e:local_rel} in stringnet space to
transform it as follows:
\begin{align*}
\tilde\theta (B_p x)&=
\sum_{k}\frac{d_k}{\dim R}\quad
\begin{tikzpicture}
\node at (0,0) {$\ast$};
 \node[morphism] (v1) at (72:2) {$\ph_1$};
 \node[morphism] (v2) at (0:2) {$\ph_5$};
 \node[morphism] (v3) at (-72:2) {$\ph_3$};
 \node[morphism] (v4) at (-144:2) {$\ph_3$};
 \node[morphism] (v5) at (144:2) {$\ph_2$};
 \draw[->] (v1)--(v5) node[pos=0.5, above]{$\scriptstyle i_1$};
 \draw[->] (v2)--(v1) node[pos=0.5, above right]{$\scriptstyle i_5$};
 \draw[->] (v3)--(v2) node[pos=0.5, below right]{$\scriptstyle i_4$};
 \draw[->] (v4)--(v3) node[pos=0.5, below]{$\scriptstyle i_3$};
 \draw[->] (v5)--(v4) node[pos=0.5, left]{$\scriptstyle i_2$};
 \draw (0,0) circle(0.7);
 \node[above right] at (30:0.8) {$k$};
 \end{tikzpicture}\\
 &=B_p^s(\tilde\theta(x))
\end{align*}
(recall that  $\A=\Rep R$, $\DD^2=\dim R$).
\end{proof}

Combining \leref{l:main1}, \leref{l:main2}, we get the statement of the
theorem. 

\begin{corollary}\label{c:renorm}
  The space $K_R(\Si, \De)$  is independent of the choice of cell
  decomposition $\De$. 
\end{corollary}

%%%%%%%%%%%%%%%%%%%%%%%%%%%%%%%%%
\section{Excited states and Turaev--Viro theory with boundary}
%%%%%%%%%%%%%%%%%%%%%%%%%%%%%%%%%
In the previous section, we constructed a Hamiltonian on the Hilbert 
space $\HK(\Si,\De)$. The Hamiltonian had a special form; it was 
expressed as a sum of local commuting projectors. We saw that the ground
state was naturally isomorphic to that in Turaev-Viro theory. In this
section we study higher eigenstates of the Hamiltonian, which are typically
called excited states. Physically, excited states are interpreted as 
``quasiparticles'' (anyons) of various types sitting on the surface $\Si$.
Excited states can also be described in Turaev-Viro theory, viewed as an
extended 3-2-1 TQFT; a particle in this language corresponds to a puncture
in the surface with certain boundary conditions. 

\subsection{Excited states in Kitaev model}
As before, let $\Si$ be a closed surface with a cell decomposition $\De$. 

Recall (see \deref{d:site}) that a site of $\De$  is a pair $s=(v,p)$ of a
vertex and incident edge. 

\begin{definition} \label{d:disjointsites}
  Two sites $s = (v, p)$ and $s' = (v', p')$ are said to be {\em disjoint}
  if $v$ is not incident to $p'$ and $v'$ is not incident to $p$ (which in
  particular implies that  $v \neq v'$ and $p \neq p'$). More generally, we
  call a collection of $n$ sites  {\em disjoint} if any two among them are
  disjoint. 
\end{definition}

The following result immediately follows from \thref{t:com_relations}.
\begin{lemma}\label{l:Daction}
  Each site \textbf{s} defines an action $\rho_s$  of  $D(R)$ on $\HK$; if
  $s,s'$ are disjoint sites, then these actions commute. 
\end{lemma}

From this perspective, the ground state $K_{R}(\Si)$ has the trivial
representation of $D(R)$ attached to \textit{every} site.  

 In physics language, a representation $V$ of $D(R)$ at a site $s$ models a
particle of type $V$ at $s$, with the trivial representation corresponding
to the absence of a particle. Thus, the ground state has no particles at
all at any site; it is called the vacuum state.

Now suppose we fix a collection of $n$ disjoint sites $S=\{s_{1}, \dots,
s_{n}\}$, $s_i=(v_i,p_i)$. For a vertex $v$, we will write $v\in S$ if $v$
is one of the vertices $v_i$, and similarly for a plaquette $p$.

Define the operator $H_S\colon \HK(\Si,\De) \to\HK(\Si,\De)$ by 
\begin{equation}
H_S =  \sum_{v \notin S}(1 - A_{v})  
    +  \sum_{p \notin S}(1 - B_{p})
\end{equation}

Let 
\begin{equation}\label{e:LL}
\LL(\Si,\De, S)=\ker(H_S)=
    \{x \in \HK\st A_v x=x\ \forall v\notin S,
                             \quad 
                     B_p x= x\ \forall p\notin S\}
\end{equation}
We think of $\LL(\Si,\De,S)$  as the space of $n$ particles fixed at sites
$s_1,\dots, s_n$ on the surface; for brevity, we will frequently drop $\Si$
and $\De$ from the notation, writing  just $\LL(S)$.  Our next goal is to
describe this space. 

% As we'll soon see the space $\LL(s_1, \dots, s_N)$ has some restrictions
% on
% the particle types (irreducible representations of $D(R)$) attached to
% each
% site. For example, it is well-known that on the sphere, $\LL(s_1)$ is one
% dimensional and in $\LL(s_1, s_2)$ the particles must have dual types.
% These are not obvious from the definitions above. 

By \leref{l:Daction}, we have an action of the algebra $D(R)^{\otimes S}$
on $\LL(S)$. Since the algebra $D(R)^{\otimes S}$ is semisimple, we can
write
\begin{equation}\label{e:protected}
\LL(s_1, \dots, s_n) =  \bigoplus_{Y_1, \dots, Y_n} 
           (Y^{*}_{1} \boxtimes Y^{*}_{2} \boxtimes \dots 
                   \boxtimes Y^{*}_{n}) \otimes 
           \M(\Si, Y_1, \dots, Y_n)
\end{equation}
where $Y_{1}, \dots, Y_{n} \in \Irr(D(R))$ are irreducible representations
of $D(R)$  and $\M(\Si, Y_1, \dots, Y_n)$ is some vector space. (Note that
$\M$ also depends on the   cell decomposition $\De$ and the set of sites
$S$;  we will usually suppress it in the notation.) The algebra
$D(R)^{\otimes S}$ acts in an obvious way on the tensor product $Y^{*}_{1}
\boxtimes \dots \boxtimes Y^{*}_{n}$ and acts trivially on the space $\M$.

The space $\M(\Si, Y_1, \dots, Y_n)$ is called the \textit{protected
subspace} in  \cite{kitaev}. It is unaffected by local
operators, as suggested above, but we can act on it (in a suitable sense),
by \textit{nonlocal} operators, such as creating, interchanging or
annihilating particles. For example, there is a natural action of the braid
group on $\M$ which, with suitable starting data, is capable of performing
universal quantum computation.

Our next goal  will be relating   the protected space
$\M(\Si, Y_1,\dots,Y_n)$ with the Turaev--Viro and stringnet  model for surfaces
with boundary.  

%%%%%%%%%%%%%%%%%%%%%%%%%%%%%%%%%%%%%%%%%
\subsection{Rewriting the protected space}
Throughout this section, $\Si, S,  \YY$ are as in the previous section.

First, recall that the
space $\LL(s_1, \dots, s_n)$ has a natural structure of a
$D(R)^{\otimes n}$ module, where the $i$-th copy of $D(R)$ acts on site
$s_i$. Since for any collection $Y_{i} \in \Irr(D(R)$, the vector space
$Y_{1}\boxtimes \dots \boxtimes Y_{n}$ also has a natural structure of
$D(R)^{\otimes n}$-module, we can define the action of a $D(R)^{\otimes
n}$ on the space  
$$ 
(Y_1 \boxtimes \dots \boxtimes Y_{n})\otimes \LL(s_1, \dots, s_n)
$$
using Hopf algebra structure of $D(R)^{\otimes n}$. 

Using the decomposition of $\LL(s_1, \dots, s_n)$ from
\eqref{e:protected}, we can extract the protected space
$\M$: 
\begin{equation}\label{e:protected2}
\M(\Si, \YY) \cong 
[(Y_1 \boxtimes \dots \boxtimes Y_{n})\otimes \LL(s_1, \dots,
s_n)]^{D(R)^{\otimes n}} 
\end{equation}

Equivalently, consider the vector space 
$$
\HK(\Si, \De, \YY)=(Y_1 \boxtimes \dots \boxtimes Y_{n})\otimes
\HK(\Si,\De)
$$
where $\HK(\Si, \De)$ is the crude Hilbert space defined in
\seref{s:crudespace}. We will graphically represent vectors in this space
by writing a vector $x_e\in R$ next to each oriented edge $e$, and also
drawing, for every site $s_i$, a green segment connecting $v$ and center of the plaquette $p$ (as in \firef{f:site}) labelled by $y_i$, as shown in \firef{f:site-vector}. 
\begin{figure}[ht]
  \begin{tikzpicture}[baseline=0cm]
  \node[below] at (0, 0) {$p_i$};
  \pentagon
  \node[above right] at (v1) {$v_i$};
  \draw[green] (0,0)--(v1) node [black, pos=0.4, left] {$\scriptstyle y_i$};
  \draw[->] (v1) -- (v2) node[pos=0.5, right]{$\scriptstyle x_n$};
  \draw[->] (v2) -- (v3);
  \draw[->] (v3) -- (v4); 
  \draw[->] (v4) -- (v5) node[pos=0.5, left]{$\scriptstyle x_2$}; 
  \draw[->] (v5) -- (v1) node[pos=0.5, above left=-2pt]{$\scriptstyle
x_1$};
  \end{tikzpicture}
 \caption{Graphical presentation of a vector in $\HK(\Si, \De, \YY)$}\label{f:site-vector}
 \end{figure}

For every vertex $v$ and $a\in R$,  define the operator 
$\tilde A_v^a\colon \HK(\Si, \De, \YY)\to \HK(\Si, \De, \YY)$ by 
$\tilde A_v^a=\id_{\YY}\otimes A_v^a$ if $v\notin \{s_1,\dots, s_n\}$ and
by the figure below if $v=v_i\in S$:

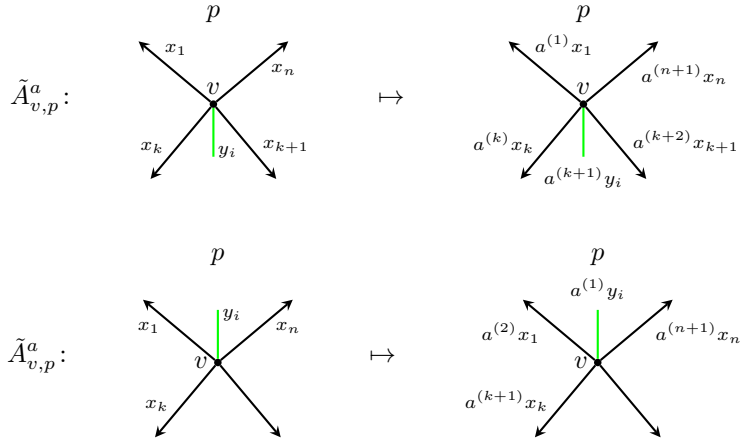
\begin{center}
$$\tilde A^a_{v,p}\colon 
\qquad
\begin{tikzpicture}[baseline=0cm]
\node[dotnode] (v) at (0,0){}; \node[above] at (v) {$v$};
\node at (0,1.2) {$p$};
\draw[->] (v)-- +(40:1.3) node[pos=0.7, below right=-2pt]{$\scriptstyle
x_n$};
\draw[->] (v)-- +(-50:1.3) node[pos=0.7, above right=-2pt] {$\scriptstyle
x_{k+1}$};
\draw[green] (v)-- +(-90:0.7) node[black,pos=0.9, right=-2pt] {$\scriptstyle
y_i$};;
\draw[->] (v)-- +(-130:1.3) node[pos=0.7, above left=-2pt]{$\scriptstyle
x_k$};
\draw[->] (v)-- +(140:1.3) node[pos=0.7, above right=-2pt]{$\scriptstyle
x_1$};
\end{tikzpicture}
\qquad
\mapsto
\qquad
\begin{tikzpicture}[baseline=0cm]
\node[dotnode] (v) at (0,0){}; \node[above] at (v) {$v$};
\node at (0,1.2) {$p$};
\draw[->] (v)-- +(40:1.3) node[pos=0.7, below right=-2pt]{$\scriptstyle
a^{(n+1)}x_n$};
\draw[->] (v)-- +(-50:1.3) node[pos=0.7, above right=-2pt] {$\scriptstyle
a^{(k+2)}x_{k+1}$};
\draw[green] (v)-- +(-90:0.7) node[black, below] {$\scriptstyle
a^{(k+1)}y_i$};;
\draw[->] (v)-- +(-130:1.3) node[pos=0.7, above left=-2pt]{$\scriptstyle
a^{(k)}x_k$};
\draw[->] (v)-- +(140:1.3) node[pos=0.7, above right=-2pt]{$\scriptstyle
a^{(1)}x_1$};
\end{tikzpicture}
$$
\end{center}

\begin{center}
$$\tilde A^a_{v,p}\colon 
\qquad
\begin{tikzpicture}[baseline=0cm]
\node[dotnode] (v) at (0,0){}; \node[left] at (v) {$v$};
\node at (0,1.4) {$p$};
\draw[->] (v)-- +(40:1.3) node[pos=0.7, below right=-2pt]{$\scriptstyle
x_n$};
\draw[->] (v)-- +(-50:1.3);
\draw[green] (v)-- +(90:0.7) node[black,pos=0.9, right=-2pt] {$\scriptstyle
y_i$};;
\draw[->] (v)-- +(-130:1.3) node[pos=0.7, above left=-2pt]{$\scriptstyle
x_k$};
\draw[->] (v)-- +(140:1.3) node[pos=0.7, below left=-2pt]{$\scriptstyle
x_1$};
\end{tikzpicture}
\qquad
\mapsto
\qquad
\begin{tikzpicture}[baseline=0cm]
\node[dotnode] (v) at (0,0){}; \node[left] at (v) {$v$};
\node at (0,1.4) {$p$};
\draw[->] (v)-- +(40:1.3) node[pos=0.7, below right=-2pt]{$\scriptstyle
a^{(n+1)}x_n$};
\draw[->] (v)-- +(-50:1.3);
\draw[green] (v)-- +(90:0.7) node[black, above] {$\scriptstyle
a^{(1)}y_i$};;
\draw[->] (v)-- +(-130:1.3) node[pos=0.7, above left=-2pt]{$\scriptstyle
a^{(k+1)}x_k$};
\draw[->] (v)-- +(140:1.3) node[pos=0.7, below left=-2pt]{$\scriptstyle
a^{(2)}x_1$};
\end{tikzpicture}
$$
\end{center}

Similarly, for any plaquette $p$ and $\al\in \Rbar$, define  the operator 
$\tilde B_p^\al\colon \HK(\Si, \De, \YY)\to \HK(\Si, \De, \YY)$ by 
$ \tilde B_p^\al= \id_{\YY}\otimes B_p^\al$ if $ p\notin
\{s_1,\dots, s_n\}$ and by the figure below if $p=p_i\in S$ (recall that
comultiplication in $\Rbar$ is given by $\De(\al)=\al''\otimes \al'$):

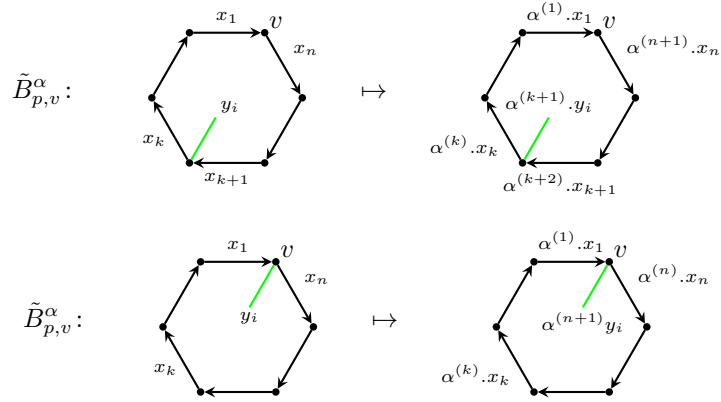
$$\tilde B^\al_{p,v}\colon 
\qquad
\begin{tikzpicture}[baseline=0cm]
 \node[dotnode] (v1) at (60:1) {};
 \node[dotnode] (v2) at (0:1) {};
 \node[dotnode] (v3) at (-60:1) {};
 \node[dotnode] (v4) at (-120:1) {};
 \node[dotnode] (v5) at (180:1) {};
 \node[dotnode] (v6) at (120:1) {};
 \node[above right=-2pt] at (v1) {$v$};
 \draw[->] (v1)--(v2) node[pos=0.5, above right] {$\scriptstyle x_n$};
 \draw[->] (v2)--(v3);
 \draw[->] (v3)--(v4) node[pos=0.5, below] {$\scriptstyle x_{k+1}$};
 \draw[->] (v4)--(v5) node[pos=0.5, below left=-2pt] {$\scriptstyle
x_{k}$};
 \draw[->] (v5)--(v6);
 \draw[->] (v6)--(v1) node[pos=0.5, above] {$\scriptstyle x_1$}; 
 \draw[green](v4)-- +(60:0.7) node[black, above right=-2pt]{$\scriptstyle
y_i$};
\end{tikzpicture}
\quad\mapsto\quad 
\begin{tikzpicture}[baseline=0cm]
 \node[dotnode] (v1) at (60:1) {};
 \node[dotnode] (v2) at (0:1) {};
 \node[dotnode] (v3) at (-60:1) {};
 \node[dotnode] (v4) at (-120:1) {};
 \node[dotnode] (v5) at (180:1) {};
 \node[dotnode] (v6) at (120:1) {};
 \node[above right=-2pt] at (v1) {$v$};
 \draw[->] (v1)--(v2) node[pos=0.5, above right] 
                       {$\scriptstyle \al^{(n+1)}.x_n$};
 \draw[->] (v2)--(v3);
 \draw[->] (v3)--(v4) node[pos=0.5, below]
{$\scriptstyle\al^{(k+2)}.x_{k+1}$};
 \draw[->] (v4)--(v5) node[pos=0.5, below left=-2pt] {$\scriptstyle
\al^{(k)}.x_{k}$};
 \draw[->] (v5)--(v6);
 \draw[->] (v6)--(v1) node[pos=0.5, above] 
   {$\scriptstyle \al^{(1)}. x_1$}; 
 \draw[green](v4)-- +(60:0.7) node[black, above=-2pt]{$\scriptstyle
\al^{(k+1)}. y_i$};
\end{tikzpicture}
$$

$$\tilde B^\al_{p,v}\colon 
\qquad
\begin{tikzpicture}[baseline=0cm]
 \node[dotnode] (v1) at (60:1) {};
 \node[dotnode] (v2) at (0:1) {};
 \node[dotnode] (v3) at (-60:1) {};
 \node[dotnode] (v4) at (-120:1) {};
 \node[dotnode] (v5) at (180:1) {};
 \node[dotnode] (v6) at (120:1) {};
 \node[above right=-2pt] at (v1) {$v$};
 \draw[->] (v1)--(v2) node[pos=0.5, above right] {$\scriptstyle x_n$};
 \draw[->] (v2)--(v3);
 \draw[->] (v3)--(v4);
 \draw[->] (v4)--(v5) node[pos=0.5, below left=-2pt] {$\scriptstyle
x_{k}$};
 \draw[->] (v5)--(v6);
 \draw[->] (v6)--(v1) node[pos=0.5, above] {$\scriptstyle x_1$}; 
 \draw[green](v1)-- +(-120:0.7) node[black, below=-2pt]{$\scriptstyle
y_i$};
\end{tikzpicture}
\quad\mapsto\quad 
\begin{tikzpicture}[baseline=0cm]
 \node[dotnode] (v1) at (60:1) {};
 \node[dotnode] (v2) at (0:1) {};
 \node[dotnode] (v3) at (-60:1) {};
 \node[dotnode] (v4) at (-120:1) {};
 \node[dotnode] (v5) at (180:1) {};
 \node[dotnode] (v6) at (120:1) {};
 \node[above right=-2pt] at (v1) {$v$};
 \draw[->] (v1)--(v2) node[pos=0.5, above right] 
                       {$\scriptstyle \al^{(n)}.x_n$};
 \draw[->] (v2)--(v3);
 \draw[->] (v3)--(v4);
 \draw[->] (v4)--(v5) node[pos=0.5, below left=-2pt] {$\scriptstyle
\al^{(k)}.x_{k}$};
 \draw[->] (v5)--(v6);
 \draw[->] (v6)--(v1) node[pos=0.5, above] 
   {$\scriptstyle \al^{(1)}. x_1$}; 
\draw[green](v1)-- +(-120:0.7) node[black, below=-2pt]{$\scriptstyle
\al^{(n+1)}y_i$};
\end{tikzpicture}
$$

It is easy to see that then the operators $\tilde A_v$, $\tilde B_p$
satisfy the relations of   \thref{t:com_relations}; in particular, for any
site $s=(v,p)$ (including the sites $s_1,\dots, s_n$), the operators
$\tilde A_v, \tilde B_p$ satisfy the relations of Drinfeld double. It
follows from the definition of  $\LL$ and \eqref{e:protected2} that 
\begin{equation}\label{e:protected3} 
\M(\Si,\De,  \YY) = 
\{x \in \HK(\Si, \De,\YY) \st \tilde A^{h}_{v}x =\tilde  B^{\bar h}_{p}x = x
\qquad \forall v,p \} 
\end{equation}

\subsection{Turaev--Viro theory surfaces with boundary} 

We recall the definition of Turaev--Viro model for surfaces with boundary,
following \ocite{balsam-kirillov}. As before, let $\A$ be a spherical
fusion category. Let $\C$ be the Drinfeld center of $\A$; as is well known,
in the example $\A=\Rep(R)$, we have $\C=\Rep(D(R))$. We have an obvious
forgetful functor $F\colon \C\to \A$ which has an adjoint $I\colon \A\to
\C$ (see details in \ocite{balsam-kirillov}). 

We will use colored graphs on surfaces where some of the lines are colored
by elements of the Drinfeld double. When drawing such graphs, we
will show objects of $Z(\A)$ by double green lines   and the
half-braiding isomorphism $\ph_Y\colon Y\otimes V\to V\otimes Y$   by
crossing as in \firef{f:crossing}.
\begin{figure}[ht]
\begin{tikzpicture}[baseline=0pt]
\draw[drinfeld center](0,2)--(1,0) node [pos=0.2, left,black] {$Y$};
\draw[overline](1,2) -- (0,0) node [pos=0.2, right] {$V$};
\end{tikzpicture}
\caption{Graphical presentation of the half-braiding $\ph_Y\colon Y\otimes
V\to V\otimes Y$, $Y\in \Obj  Z(\A)$, $V\in \Obj \A$}\label{f:crossing}
\end{figure}
 
Now  let $\Si_0$ be a surface with $n$ boundary
components, together with a choice of marked point $p_a$ on each
boundary component.  Consider the new
surface $\Si$ obtained by gluing to $\Si_0$ $n$ copies of the standard
2-disk. This is a closed surface; moreover, each cell decomposition $\De$
of $\Si_0$ gives rise to a cell decomposition of $\Si$ obtained by adding to
$\De$ each of the glued disks as a 2-cell. These cells will be called {\em
embedded disks}. 

We can now define the state space for such a surface. Namely, let $l$ be a
coloring of edges of $\De$ by simple objects and let $\YY=\{Y_1,\dots,
Y_n\}$ be a collection of objects of $\C$, one object for each boundary
component of $\Si_0$.  Then we define the state space 
$$
\HTV(\Si_0,\De, \YY,l)=\bigotimes_C  H_{TV}(C,l)
$$
where the product is over all 2-cells of $\De$ (including the embedded
disks) and
$$
H_{TV}(C,l)=\begin{cases}
\<Y_a,l(e_1), l(e_2),\dots, l(e_n)\>&\quad
                  C=D_a \text{ -- an embedded disk}\\
\<l(e_1), l(e_2),\dots, l(e_n)\>&\quad C \text{ -- an ordinary 2-cell of
$\N$}
\end{cases}
$$
where $e_1,e_2,\dots$ are edges of $C$ traveled counterclockwise;
for the embedded disks, we also require that we start with the
marked point $p_a$; for ordinary 2-cells of $\De$ the choice of
starting point is not important.

As before, we now define
\begin{equation}\label{e:ex_state_sum}
  \HTV(\Si_0,\De, \YY)=\bigoplus_{l}\bigotimes_{C}\HTV(C, l)
\end{equation}
where $C$ runs over the set of all 2-cells (including the embedded disks)
and  the sum is taken over all equivalence classes of colorings 
$l$ of edges of $\De$. 

It has been shown in \ocite{balsam-kirillov} that for a suitably defined
notion of a cobordism between such surfaces with embedded disks, every
cobordism $M\colon \Si_1\to\Si_2$ (together with a cell decomposition
extending the cell decompositions of $\Si_1,\Si_2$) gives rise to a
linear operator 
$$
Z_{TV}(M)\colon H(\Si_1,\De, \YY)
     \to H(\Si_2,\De, \YY)
$$
which does not depend on the choice of the cell decomposition of $M$, so
that composition of cobordisms  corresponds to composition of linear
operators. Thus, we can repeat the same steps as before and define TV
theory for surfaces with boundary by 
$$
\ZTV(\Si_0, \YY)=\im(A)
$$
where $A=\ZTV(\Si_0\times I)\colon \HTV(\Si,\De, \YY)\to
\HTV(\Si,\De, \YY)$.

It has been shown in \ocite{balsam-kirillov} that this defines a 3-2-1
TQFT; in particular, so defined vector space does not depend on the choice
of cell decomposition $\De$. 

Moreover, it is possible to compute this vector space explicitly. For
example, if $\Si_0=S^2$ is sphere with $n$ boundary components, then
$$
\ZTV(\Si_0, Y_1,\dots, Y_n)\cong
\Hom_{\C}(\one, Y_1\dotimes  Y_n.)
$$

\subsection{Stringnet for surfaces with boundary} 
We can now describe the stringnet model as an extended theory, in which we
allow surfaces with boundary. We give an overview of the theory, referring
the reader to \ocite{stringnet} for a detailed description. 

Recall that given a spherical category $\A$, we defined the notion of a
colored graph $\Ga$ on an oriented  surface $\Si_0$. 
For a surface with boundary, we consider colored graphs which  may
terminate on the boundary, and the legs terminating on the boundary should
be colored by objects of $\A$. Thus,  every colored graph $\Ga$ defines
a collection of points $B=\{b_1,\dots, b_n\}\subset \del \Si_0$ (the
endpoints of the legs of $\Ga$) and a collection of objects $V_b\in \Obj\
\A$ for every $b \in B$: the colors of the legs of $\Ga$ taken with
outgoing orientation. We will denote the pair $(B, \{V_b\})$ by
$\VV=\Ga\cap \del\Si$ and call it {\em boundary value}. Similar to the
closed case, we can define, for a fixed boundary value $\VV$, the
stringnet space
$$
\Hs(\Si_0, \VV)=\left(\parbox{6cm}{formal combinations of colored graphs\\ 
 with boundary value $\VV$}\right)/\text{local relations}
$$

It was shown in \ocite{stringnet} that boundary conditions actually form a
category $\Chat(\del \Si_0)$  so that $\Hs(\Si_0, \VV)$ is functorial in $\VV$.
Moreover, if we denote by $\C(\del \Si_0)$ the pseudo-abelian completion of
this category, then one has an equivalence 
\begin{align*}
J\colon \C(S^1)&\simeq \C\\
\{V_1,\dots,V_n\}&\mapsto I(V_1\dotimes V_n)
\end{align*}
where $\C=Z(\A)$ is the Drinfeld center of $\A$ and $I\colon \A\to \C$
is the adjoint of the forgetful functor $Z(\A)\to \A$. Thus, if $\del\Si_0$
is a union of $n$ circles, then a choice of parametrization $\psi\colon
\del\Si_0\simeq S^1\sqcup \dots\sqcup S^1$ gives rise to an equivalence of
categories $\C(\del\Si_0)\simeq \C^{\boxtimes n}$.

Since any functor $\Chat\to \Vect$ naturally extends to a functor of the
pseudoabelian completion $\C\to\Vect$, we can define the stringnet space
$\Hs(\Si_0, \YY)$ for any $\YY\in \C(\del\Si_0)$. Equivalently, given a surface 
$\Si_0$ together with a parametrization $\psi$ of the boundary components,
we can define the vector space  $\Hs(\Si_0, \psi, \YY)$, where
$\YY=\{Y_1,\dots, Y_n\})$, $Y_a\in Z(\A)$.

The space $\Hs(\Si_0, \psi, \YY)$ admits an alternative definition.
Namely, let $\Si$ be the closed surface obtained by gluing to $\Si_0$ a
copy of the standard 2-disk $D$ along each boundary circle $(\del \Si_0)_a$
of $\Si_0$, using parametrization $\psi_a$. So defined, the surface comes with a
collection of marked points $p_a=\psi_a^{-1}(p)$, where $p=(1,0)$ is the
marked point on $S^1$. Moreover, for every  point $p_a$ we also have a
distinguished  ``tangent direction'' $v_a$ at $p_a$ (in PL setting, we
understand it as a germ of an arc staring at $p_a$), namely the direction
of the radius connecting $p$ with the center of the disk $D$. We will refer
to the collection  $(\Si, \{p_a\}, \{v_a\})$ as an {\em extended
surface}. It is easy to see that given  $(\Si, \{p_a\}, \{v_a\})$, the
original surface $\Si_0$ and parametrizations $\psi_a$ are defined uniquely
up to a contractible set of choices. 

For such an extended surface and a choice of collection of objects
$\YY=\{Y_1,\dots, Y_n\})$, $Y_a\in Z(\A)$,  define 
\begin{equation}\label{e:Hs-extended}
\Hhs(\Si, \YY)=\VGr'(\Si,  \YY)/(\text{Local relations})
\end{equation}
where $\VGr'(\Si, \YY)$ is the vector space of
formal linear combinations of colored graphs on $\Si$ such that each
colored graph has an uncolored one-valent vertex at each point $p_a$, with
the corresponding edge coming from direction $v_a$ (i.e., in some
neighborhood of $p_a$, the edge coincides with the corresponding arc) and
colored by the object $F(Y_a)$ as shown in  \firef{f:marked_point}, and
local relations are defined in the same way  as before: each embedded disk
$D\subset \Si$ {\bf not containing the special points} $p_a$ gives rise to
local relations.  
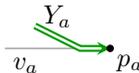
\begin{figure}[ht]
$$
%%%%%%%%%%%%%
\begin{tikzpicture}
\draw[thin, gray] (-1.4,0)--(0,0) node[pos=0.2, below, black] {$v_a$};
\draw[drinfeld center, ->] (-1,0.3)--(-0.4,0) node[pos=0.5, above] {$Y_a$}
                               --(0,0);
\node[dotnode, label=-45:$p_a$] at (0,0) {};
\end{tikzpicture}
%%%%%%%%%%%%%
$$
\caption{Colored graphs in a neighborhood of marked point}
\label{f:marked_point}
\end{figure}

The following lemma is a reformulation of results of \ocite{stringnet}.

\begin{lemma}\label{l:stringnet-bdry}
Let $\Si_0$ be a compact surface with $n$ boundary components,  $\psi\colon
\del\Si_0\simeq S^1\sqcup \dots\sqcup S^1$ --- a parametrization of the
boundary, and $\YY=\{Y_1,\dots, Y_n\})$,  $Y_a\in Z(\A)$ --- a choice of
boundary conditions. Then one has a canonical isomorphism 
$$
\Hs(\Si_0, \psi, \YY)=\{x \in \Hhs(\Si, \YY)\st B^{str}_{p_a}x =x \ \forall a\}
$$
where for each marked point $p_a$, the operator $B^{str}_{p_a}$ is defined by 
\begin{figure}[ht]
$$
\begin{tikzpicture}
\draw[drinfeld center, ->] (-1.2,0)--(0,0);
\node[dotnode, label=-45:$p_a$] at (0,0) {};
\end{tikzpicture}
%%%%%%%%%%%%%%%%%%
\quad\mapsto\quad  \sum_i \frac{d_i}{\DD^2}\quad
%%%%%%%%%%%%%%%%%%
  \begin{tikzpicture}
\draw[drinfeld center, ->] (-1.2,0)--(0,0);
\node[dotnode, label=-45:$p_a$] at (0,0) {};
    \draw[overline] circle(0.6);
    \node at (30:0.8) {$i$};
  \end{tikzpicture}
$$
\caption{Operator $B_p$ for a marked point}\label{f:B_p_marked}
\end{figure}

\end{lemma}

The following is the main result of  \ocite{stringnet}.
\begin{theorem}\label{t:TVLWbdry}
Let $\Si_0$ be a compact surface with $n$ boundary components,  $\psi\colon
\del\Si_0\simeq S^1\sqcup \dots\sqcup S^1$ --- a parametrization of the
boundary. Then for  any $\YY=\{Y_1,\dots,  Y_n\}\in Z(\A)^{\boxtimes n}$,
one has a  canonical functorial  isomorphism
  $$
    \ZTV(\Si_0,\psi, \YY)\cong \Hs(\Si_0,\psi,\YY).
  $$
where, as before, $\Si$ is obtained from $\Si_0$ by gluing disks along the
boundary. 
\end{theorem}

As before, we will need a more detailed construction of the isomorphism of this theorem, parallel to the description for closed surfaces given in 
\thref{t:TVLW2}.  Namely, let $\De_0$  be a cell decomposition of $\Si_0$ such that for every boundary component $(\del \Si_0)_a$, the corresponding marked point $p_a=\psi^{-1}(0,1)$ is a vertex of $\De_0$. By adding to $\De_0$ a disk $D_a$ for each boundary component, we get a cell decomposition $\De$  of closed surface $\Si$.

Let $\Si-\De^0$ be the surface
with punctures obtained by removing from $\Si$ all vertices of $\De$ (this includes the marked points $p_a$).
Let $\Hhs(\Si-\De^0, \YY)$ be the stringnet space defined by boundary condition of \firef{f:marked_point} near puncture $p_a$ (and trivial boundary condition near all other punctures). Then one has the following results.

\begin{theorem}\label{t:TVLWbdry2}
  \par\noindent
  \begin{enumerate}
    \item One has  an isomorphism
$$
\Hs(\Si, \YY)\simeq \{x\in\Hhs(\Si-\De^0, \YY)\st B^{str}_p x=x\ \forall p\in \De^0\}\subset
\Hhs(\Si-\De^0, \YY)$$
where $B^s_p\colon\HsD\to\HsD$ is the operator which adds to a colored graph a
small loop around puncture $p$ as shown in \firef{f:B_p}, \firef{f:B_p_marked}

    \item One has a natural isomorphism 
    $\HTV(\Si_0,\De_0, \YY)\simeq \Hhs(\Si-\De^0, \YY)$
    \item Under the isomorphism of the previous part, the operator
associated to the cylinder  $Z_{TV}(\Si\times I)\colon \HTV\to\HTV$ is
identified with the projector $B^s=\prod_p B_p\colon \Hhs(\Si-\De^0, \YY)\to\Hhs(\Si-\De^0, \YY)$.
  \end{enumerate}
\end{theorem}

The proof of this theorem can be found in \ocite{stringnet}; obviously, it
implies \thref{t:TVLWbdry}.

%%%%%%%%%%%%%%%%%%%%%%%%%%%%%%%%%%%%%%%%%%%%%%%%%%%%%%%%%%%%
\section{Comparison of Kitaev, Turaev--Viro and Levin-Wen models for
surfaces with boundary}\label{s:main-boundary}

In this section, we establish the relation between the protected space for
Kitaev's model and the Turaev--Viro (and thus the Levin-Wen) space for
surfaces with boundary, extending \thref{t:main1} to surfaces with
boundary. 
\subsection{Statement of the main theorem}\label{s:main2}
As before, we fix a semisimple Hopf algebra $R$ over $\CC$ and denote
$\A=\Rep R$. As was mentioned before, in this  case we also have a
canonical equivalence of categories $\Rep(D(R))\cong Z(\A)$. 

Throughout the section, we fix a choice of a compact oriented surface $\Si$
(without boundary) and a cell  decomposition $\De$ of $\Si$. We also fix a
finite collection of disjoint  sites $S=\{s_1, \dots, s_n\}$ and a finite
collection $\YY=\{Y_1, \dots, Y_n\}$ of irreducible representations of
$D(R)$. 

We denote by $\De^*$ the dual cell decomposition of $\Si$. Then each site $s_i=(v_i,p_i)$ defines a cell $D_i$ (containing $v_i$) of $\De^*$ and a marked point $p_i$ on the boundary of $D_i$.

Denote by $\Si_0$ the surface with boundary and marked points,  obtained by removing from  $\Si$ the interiors of $D_1, \dots, D_n$. Clearly, in this situation $\Si$ can be obtained from $\Si_0$ by gluing the disks $D_1, \dots, D_n$.

\begin{theorem} \label{t:main2}
Let $\Si$, $\Si_0$, $\YY$ be as above. Then one has a
  canonical functorial  isomorphism
  $$
    \M(\Si,\De, \YY)\cong \ZTV(\Si_0,\YY)\cong \Hs(\Si_0,\{p_i\}, \YY),
  $$
 where $\M(\Si,\De,  \YY)$ is the protected space defined by
  \eqref{e:protected}.
\end{theorem}

The following example is instructive.
\begin{example}
Let $\Si$ be the sphere with $n$ sites  labeled by $Y_1, \dots, Y_n$.
Then $Z_{TV}(\Si_0, Y_1, \dots, Y_n) \cong \Hom_{D(R)}(\one, Y_1 \otimes
\dots \otimes Y_n)$ \cite{balsam-kirillov}. It follows that for $n = 1$,
$\M(Y)$ is one-dimensional if $Y$ is trivial one-dimensional representation
of $D(R)$, and $\M(Y)=0$ if $Y$ is non-trivial irreducible representation
of $D(R)$; thus, $\LL(s_1)=\M(s_1, \one)$ is one dimensional, i.e. there
are no  single particle excitations on the sphere. For $n = 2$,
$Z_{TV}(\Si_0, Y,Z) =\Hom_{D(R)}(\one, Y \otimes Z) = 0$, unless $Z \cong
Y^*$. It follows that
two-particle excitations on the sphere consist of a particle of type $Y$ at
one site and a particle of type $Y^*$ at another site. 
\end{example}

The proof of the theorem occupies the rest of this section. We begin with
some preliminary results.

\subsection{Lemma on Haar integral}
We will need the following technical lemma.

\begin{lemma} \label{l:mainkl}
    Let $Y$ be a representation of $D(R)$, and let $\hbar\in \Rbar$ be the
Haar integral of $\Rbar$. 

Consider the map 
\begin{align*}
    Y\otimes R&\to Y\\
    y\otimes r&\mapsto \hbar''.y \<\hbar', r\>
\end{align*}
where $\la.y$ stands for the action of $\Rbar$ on $Y$. 

Then under the isomorphism $R\simeq \bigoplus V_i\otimes V_i^*$, this map
is identified with the map pictured below
\begin{equation}\label{e:crossing3}
\sum_i \frac{d_i}{\dim R}\qquad
\begin{tikzpicture}[baseline=2cm]
\draw[drinfeld center, ->](2,4)--(4, 2)  node[left, pos=0.2] {$Y$} -- (2,0)
;
\draw[->] (3,4)-- +(0,-2);
\draw[overline] (3,2) arc (-180:0:1.5) -- (6,4) node[right, pos=0.9] {$V_i$};
 \end{tikzpicture}
\end{equation}
where the upper crossing is just the permutation of factors \textup{(}note that it is not a morphism of modules\textup{)}. 
\end{lemma}
\begin{proof}
    By definition of the $R$-matrix \eqref{e:Rmatrix}, the map \eqref{e:crossing3} is given
by 
\begin{align*}
    y\otimes v\otimes f&\mapsto 
      \frac{d_i}{\DD^2}\sum_{\al}x^\al y\otimes x_\al v\otimes f
      \mapsto \frac{d_i}{\DD^2}\sum_{\al}x^\al y\otimes \<x_\al v, f\>\\
      &=\sum_\al x^\al y \<\hbar, x_\al v\otimes f\>
\end{align*}
where $v\in V_i$, $f\in V_i^*$, and $x^\al, x_\al$ are dual bases in
$\Rbar, R$. 

Since for any $\la\in \Rbar$, we have $\<\la, xr\>=\<\la',x\>
\<\la'', r\>$,  this can be rewritten as 
$$
\sum_\al x^\al y \<\hbar', x_\al\>\<\hbar'', v\otimes f\>
=\hbar'  y \<\hbar'', v\otimes f\>
$$
Since $\De(\hbar)$ is symmetric ($\hbar'\otimes \hbar''=\hbar''\otimes \hbar'$), we get the statement of the lemma. 
\end{proof}

Combining this with the formula for multiplication and comultiplication in $R$ under the isomorphims $R\simeq \bigoplus V_i\otimes V_i^*$, we get the following corollary, generalizing \leref{l:hbar2}.
\begin{corollary}\label{c:mainkl}
Consider the map 
\begin{align*}
Y\otimes R^{\otimes n}&\to Y\otimes R^{\otimes n}\\
y\otimes x_n\dotimes x_1&\mapsto 
    \hbar^{(n+1)}y\otimes \hbar^{(n)}.x_n\dotimes \hbar^{(1)}.x_1\\
    &=\hbar''y\otimes \<\hbar',S(x'_n\dots x'_1)\> x''_n\dotimes x''_1
\end{align*}
    
Then under the isomorphism $R\simeq \bigoplus V_i\otimes V_i^*$, this map
is identified with the map pictured below
\begin{align*}
&y\otimes x_n\dotimes x_1 \mapsto 
\sum_{i_1, \dots, i_n,j_1, \dots, j_n, k}    \frac{d_{i_1}\dots d_{i_n}
d_k}{\dim R}\sum_{\al, \be, \dots}\\
&
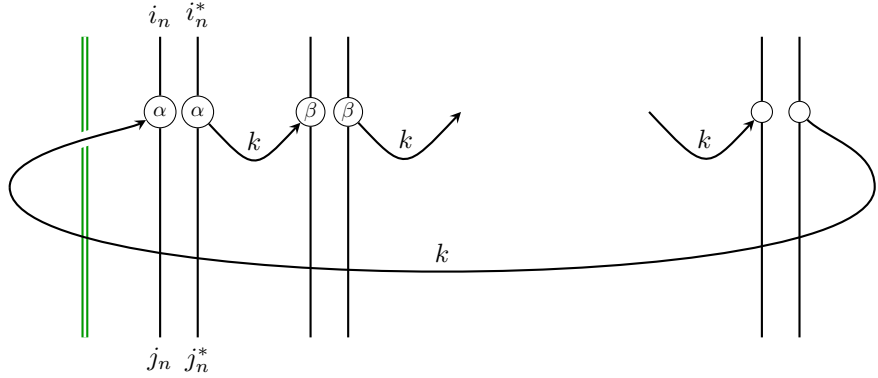
\begin{tikzpicture}
\coordinate (ytop) at (0,1); \coordinate (ybot) at (0,-3);
\draw[drinfeld center](ytop)--(ybot);
%%%%%%%%%%
\node[morphism] (v1) at (1,0) {$\al$};
\node[morphism] (w1) at (1.5,0) {$\al$};
\node[above] at (v1|-ytop) {$i_n$};
\node[above] at (w1|-ytop) {$i^*_n$};
\node[below] at (v1|-ybot) {$j_n$};
\node[below] at (w1|-ybot) {$j^*_n$};
\draw (v1|-ytop) -- (v1) -- (v1|-ybot);
\draw (w1|-ytop) -- (w1) -- (w1|-ybot);
%%%%%%%%%%
\node[small_morphism] (v2) at (3,0) {$\be$};
\node[small_morphism] (w2) at (3.5,0) {$\be$};
\draw (v2|-ytop) -- (v2) -- (v2|-ybot);
\draw (w2|-ytop) -- (w2) -- (w2|-ybot);
\draw[->] (w1) .. controls (2.25, -0.8).. (v2) node[pos=0.5, above] {$k$};
\draw[->] (w2) .. controls (4.25, -0.8).. (5,0) node[pos=0.5, above] {$k$};
%%%%%%%%%%
\node[small_morphism] (vn) at (9,0) {};
\node[small_morphism] (wn) at (9.5,0) {};
\draw (vn|-ytop) -- (vn) -- (vn|-ybot);
\draw (wn|-ytop) -- (wn) -- (wn|-ybot);
\draw[->] (7.5,0) .. controls (8.25, -0.8).. (vn) node[pos=0.5, above] {$k$};
%%%%%%%%%%
\coordinate (l) at (-1,-1);
\coordinate (r) at (10.5,-1);
\draw[<-, overline] (v1) 
          .. controls +(210:0.5) and +(90:0.5) ..
      (l);
\draw (l) 
          .. controls +(-90:1.5) and +(-90:1.5) ..
      (r) node[pos=0.5, above] {$k$}
          .. controls +(90:0.5) and +(-45:0.5) ..
      (wn);
\end{tikzpicture}\end{align*}

\end{corollary}

\subsection{Proof of the main theorem}
We can now complete the proof of \thref{t:main2}, by combining results of the  two previous subsections.
 
By \eqref{e:protected3}, the space $\M$ can be obtained from the
space $\HK(\Si, \De,\YY)$ by applying projectors $A_v, B_p$. Let us
consider the intermediate space obtained by $A_v$ projectors only:
 
$$\HA(\Si, \De,\YY) =
\{x \in \HK(\Si, \De,\YY) \st \tilde A^{h}_{v}x = x \qquad \forall v \}
$$

\begin{lemma}\label{l:main3}
We have an isomorphism 
$$ 
\HA(\Si, \De,\YY) \cong \HTV(\Si, \De^*, \YY^*)\cong\Hhs(\Si-\De^{*0}, \YY)
$$
where the space on the right is the string net space on $\Si$, with the
centers of plaquettes removed, and boundary condition $Y_i$ at site $s_i$
(cf. \eqref{e:Hs-extended}). 
\end{lemma}
\begin{proof}
The proof repeats with necessary changes the proof of \leref{l:main1}.

Namely, the same arguments as in the proof of \leref{l:main1} show that 
$$
\HA(\Si, \De,\YY)=\bigoplus_l\bigotimes_v H_A(v,l)
$$
where the product is over all vertices of $\De$ and 
$$
H_v(v,l)=\begin{cases}
          \<l(e_1), \dots, l(e_n)\>, \quad &v\notin S\\
          \<Y_i,l(e_1), \dots, l(e_n)\>, \quad &v=v_i\in S\\
        \end{cases}
$$
where $e_1, \dots, e_n$ are edges starting at $v$, in counterclockwise order. Thus, we see that we have a natural isomorphism
$H_A(v,l)=H_{TV}(C_v, l)$, where $C_v$ is the 2-cell of the dual cell decomposition $\De^*$  corresponding to $v$. Note that it also holds in the case when $v\in S$, in which case $C_v$ is the embedded disk $D_i$.

Thus, we have  a natural isomorphism
$$
\theta\colon \HA(\Si, \De,\YY)\simeq \HTV(\Si,\De^*, \YY)
$$

As in the proof of \leref{l:main1}, we choose to rescale it to define 
$$
\tilde \theta=\sqrt{\mathbf{d}_l} \theta
$$

\firef{f:HA2} shows the composition map
$\HA(\Si,\De,\YY)\xxto{\tilde\th} \HTV(\Si, \De^*,\YY)
\simeq \Hhs(\Si-\De^{*0}, \YY)$.

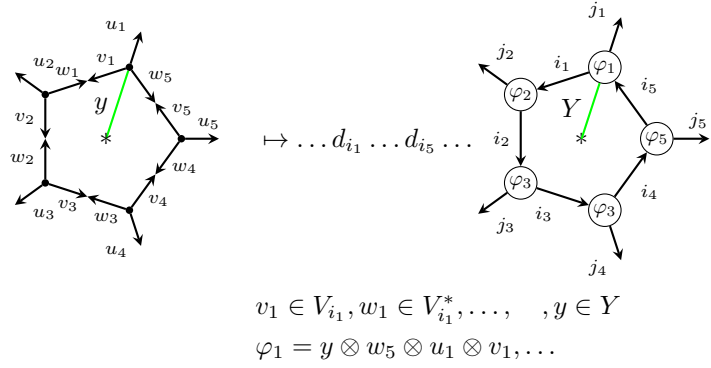
\begin{figure}[ht]
\begin{align*}
\begin{tikzpicture}
\pentagon
\draw[green](v1)--(0,0) node[pos=0.7, above left=-2pt, black] {$y$};
\node at (0,0) {$\ast$};
\coordinate (m1) at ($(v1)!.5!(v2)$); %midpoint of a side
\coordinate (m2) at ($(v2)!.5!(v3)$);
\coordinate (m3) at ($(v3)!.5!(v4)$);
\coordinate (m4) at ($(v4)!.5!(v5)$);
\coordinate (m5) at ($(v5)!.5!(v1)$);
\draw[->](v1)--  +(72:0.5) node[pos=0.7, above left]{$\scriptstyle u_1$};
\draw[->](v5)--  +(144:0.5) node[pos=0.7, above right]{$\scriptstyle u_2$};
\draw[->](v4)--  +(-144:0.5) node[pos=0.7, below right]{$\scriptstyle u_3$};
\draw[->](v3)--  +(-72:0.5) node[pos=0.7, below left]{$\scriptstyle u_4$};
\draw[->](v2)--  +(0:0.5) node[pos=0.7, above]{$\scriptstyle u_5$};
\draw[->] (v1)--(m1) node[pos=0.5, above right=-2pt]{$\scriptstyle w_5$};
\draw[->] (v2)--(m1) node[pos=0.5, above right=-2pt]{$\scriptstyle v_5$};
\draw[->] (v2)--(m2) node[pos=0.5, below right=-2pt]{$\scriptstyle w_4$};
\draw[->] (v3)--(m2) node[pos=0.5, below right=-2pt]{$\scriptstyle v_4$};
\draw[->] (v3)--(m3) node[pos=0.5, below]{$\scriptstyle w_3$};
\draw[->] (v4)--(m3) node[pos=0.5, below]{$\scriptstyle v_3$};
\draw[->] (v4)--(m4) node[pos=0.5, left]{$\scriptstyle w_2$};
\draw[->] (v5)--(m4) node[pos=0.5, left]{$\scriptstyle v_2$};
\draw[->] (v5)--(m5) node[pos=0.5, above]{$\scriptstyle w_1$};
\draw[->] (v1)--(m5) node[pos=0.5, above]{$\scriptstyle v_1$};
\end{tikzpicture}
\quad 
&\mapsto
\dots d_{i_1}\dots d_{i_5}\dots 
\begin{tikzpicture}
 \node[morphism] (v1) at (72:1) {$\ph_1$};
 \node[morphism] (v2) at (0:1) {$\ph_5$};
 \node[morphism] (v3) at (-72:1) {$\ph_3$};
 \node[morphism] (v4) at (-144:1) {$\ph_3$};
 \node[morphism] (v5) at (144:1) {$\ph_2$};
\node at (0,0) {$\ast$};
\draw[green] (v1)--(0,0) node[pos=0.7, above left=-2pt, black] {$Y$};
\draw[->](v1)--  +(72:0.7) node[pos=0.7, above left]{$\scriptstyle j_1$};
\draw[->](v5)--  +(144:0.7) node[pos=0.7, above right]{$\scriptstyle j_2$};
\draw[->](v4)--  +(-144:0.7) node[pos=0.7, below right]{$\scriptstyle j_3$};
\draw[->](v3)--  +(-72:0.7) node[pos=0.7, below left]{$\scriptstyle j_4$};
\draw[->](v2)--  +(0:0.7) node[pos=0.7, above]{$\scriptstyle j_5$};
 \draw[->] (v1)--(v5) node[pos=0.5, above]{$\scriptstyle i_1$};
 \draw[->] (v5)--(v4)  node[pos=0.5, left]{$\scriptstyle i_2$};
 \draw[->] (v4)--(v3)  node[pos=0.5, below left]{$\scriptstyle i_3$};
 \draw[->] (v3)--(v2)  node[pos=0.5, below right]{$\scriptstyle i_4$};
 \draw[->] (v2)--(v1)  node[pos=0.5, above right]{$\scriptstyle i_5$};
\end{tikzpicture}\\
& v_1\in V_{i_1}, w_1\in V_{i_1}^*, \dots, \quad, y\in Y \\
& \ph_1=y \otimes w_5\otimes u_1\otimes v_1, \dots
\end{align*}
\caption{Isomorphism $\HA\simeq\Hs_{\De^*}$. Asterisk $\ast$
shows the puncture obtained by removing a vertex of $\De^*$.}\label{f:HA2}
\end{figure}

\end{proof}
\begin{lemma}\label{l:main4}
Under the isomorphism of the previous lemma, the operators $B_p$ of Kitaev's model (for all $p$, including $p\in S$) are identified with the operators $B_p^s$ of stringnet model.
\end{lemma}

\begin{proof}
For $p\notin S$, the proof is the same as in \leref{l:main2}. For $p\in S$, it follows from \coref{c:mainkl}.
\end{proof}

Taken together, these two lemmas immediately imply \thref{t:main2}.

%
%
%\section{Appendix: Drinfeld double}
%\begin{definition}
%The \textit{Drinfeld double} $D(R)$ of $R$ is a Hopf algebra constructed as follows. As a vector space, $D(R) \cong \Rbar \otimes R$, with
%\begin{enumerate}
%\item Multiplication: $(\alpha \otimes x)\cdot(\beta \otimes y)) = 
%\alpha \tilde \beta \otimes \tilde x y$, where $\tilde \be\in \Rbar$, $\tilde x\in R$ are defined by 
%\begin{equation}\label{e:double1}
%\begin{aligned}
%&\<\tilde \be, z\>=\<\be,S^{-1}(x''')zx'\>\\
%&\tilde x=x''
%\end{aligned}
%\end{equation}
%
%\item Unit: $1_{D(R)} = 1_{\Rbar} \otimes 1_{R}$
%\item Comultiplication: $(x\otimes \alpha \mapsto (x'\otimes \alpha') 
%\otimes (x''\otimes \alpha'')$
%\item Counit: $(x\otimes \al) \mapsto \epsilon_R(x)\epsilon_{\Rbar}(\al)$
%\item Antipode: 
%$$
%S(\al\otimes x)=\tilde{\be}\otimes \tilde{y}
%$$
%where  $\be=S(\al)$, $y=S(x)$, and $\tilde \be, \tilde y$ are 
%defined by \eqref{e:double1}.
%\end{enumerate}
%\end{definition}

%%%%%%%%%%%%%%%%%%%%%%%%%%%%%%%%%%%%%%%%%%%%%%%%%%%%%%%%%%%%%%
\end{document}